\documentclass[11pt,reqno]{amsart}
\usepackage{amsmath,amscd,amssymb,amsfonts,amsthm,courier,relsize,bm}
\usepackage{hyperref,enumerate,mathrsfs,mathtools,slashed}
\usepackage[backend=bibtex,style=alphabetic,maxbibnames=10]{biblatex} 
\renewbibmacro{in:}{} 
\usepackage{xcolor}  
\hypersetup{
    colorlinks,
    linkcolor={red!50!black},
    citecolor={blue!70!black},
    urlcolor={blue!80!black}
}

\usepackage{tikz-cd}

\newtheorem{theorem}{Theorem}[section]
\newtheorem{corollary}[theorem]{Corollary}
\newtheorem{proposition}[theorem]{Proposition}

\newtheorem{lemma}[theorem]{Lemma}
\theoremstyle{definition}    
\newtheorem{definition}[theorem]{Definition}
\theoremstyle{remark}

\newtheorem{remark}[theorem]{Remark}

\newtheorem{example}[theorem]{Example}

\newcommand{\ignore}[1]{}
\newcommand{\ol}[1]{\overline{#1}}

\newcommand{\ti}[1]{\widetilde{#1}}

\newcommand{\mf}[1]{\mathfrak{#1}}

\newcommand{\tn}[1]{\textnormal{#1}}

\def\d{\ensuremath{\mathrm{d}}}

\def\A{\ensuremath{\mathcal{A}}}

\def\D{\ensuremath{\mathcal{D}}}

\def\I{\ensuremath{\mathcal{I}}}
\def\J{\ensuremath{\mathcal{J}}}

\def\M{\ensuremath{\mathcal{M}}}
\def\N{\ensuremath{\mathcal{N}}}

\def\R{\ensuremath{\mathcal{R}}}
\def\S{\ensuremath{\mathcal{S}}}

\def\U{\ensuremath{\mathcal{U}}}

\def\bR{\ensuremath{\mathbb{R}}}

\def\bZ{\ensuremath{\mathbb{Z}}}

\def\Hom{\ensuremath{\textnormal{Hom}}}

\def\supp{\ensuremath{\textnormal{supp}}}

\def\ev{\ensuremath{\textnormal{ev}}}

\def\pr{\ensuremath{\textnormal{pr}}}
\def\dim{\ensuremath{\textnormal{dim}}}
\def\Sym{\ensuremath{\textnormal{Sym}}}

\def\gr{\ensuremath{\textnormal{gr}}}

\def\gr{\ensuremath{\tn{gr}}}

\def\Spec{\ensuremath{\tn{Spec}}}

\title{Weightings and Lie groupoids}
\author{Yiannis Loizides}
\begin{document}
\sloppy
\maketitle

\begin{abstract}
This note originated with a mini-course taught at KU Leuven by the author in June 2026. We review the definition of weightings (or quasi-homogeneous structures), the weighted normal bundle, and the weighted deformation space. We emphasize the algebraic approach and functorial properties, using sheaves of ideals and the character spectrum for sheaves of algebras. We review the definition of a multiplicative weighting on a Lie groupoid $G\rightrightarrows M$, and give a self-contained proof of the three-way equivalence between (i) multiplicative weightings of $G$ along $M$, (ii) Lie filtrations of the Lie algebroid $A$ of $G$, (iii) dg weightings of the NQ manifold $A[1]$ along $M$. The corresponding weighted deformation space is a weighted variant of the tangent/adiabatic groupoid and is used in the study of hypoelliptic operators.
\end{abstract}

\tableofcontents

\section{Introduction}
Weightings of manifolds were introduced by Meinrenken and the author in \cite{LMweightings}. Quasi-homogeneous structures on manifolds, introduced by Melrose in unpublished lecture notes \cite{melrose-corners-notes} and further studied by Behr \cite{behr-quasi-homog}, is an equivalent notion. Weightings globalize, to the setting of manifolds, the technique of assigning weights to local coordinates, typically in a manner adapted to some geometric application. A weighting of a manifold $M$ along a submanifold $N$ is a coherent way of assigning weights in $\bZ_{>0}$ to local coordinates that vanish on the submanifold $N$: the `standard weighting' corresponds to setting all of these weights equal to $1$; other weightings result from assigning weights larger than $1$ to some coordinates. A weighting determines a filtration of the normal bundle $\nu(M,N)$ to $N$ by subbundles, but is not completely determined by this data; further non-linear information is involved. Applications of weightings and closely related weighted Euler-like vector fields include to normal form theorems \cite{meinrenken2021euler}, singular symplectic quotients \cite{zimhony2024commutative}, linearization of vector fields \cite{qiu-weighted-linearization}, configuration spaces \cite{gootjes-dreesbach}, functorial resolution of singularities by weighted blow-ups (cf. \cite{brais2025streamlining} and references therein), and resolutions of Poisson structures \cite{lapointe2026weighted}. 

One of the original motivations for introducing weightings in \cite{LMweightings}, and pursued further in \cite{LMsingularlie}, was to better understand and generalize constructions appearing in the study of filtered manifolds. A \emph{filtered manifold} is a manifold $M$ equipped with a filtration $0_M\subset A_1\subset A_2\subset \cdots \subset A_r=TM$ of its tangent bundle by subbundles, such that $[\Gamma(A_i),\Gamma(A_j)]\subset \Gamma(A_{i+j})$; we will refer to $A_\bullet$ as a \emph{Lie filtration} of $TM$. An example of a Lie filtration with $r=2$ appears in contact geometry, with $A_1=H$ the contact hyperplane distribution and $A_2=TM$. More generally, filtered manifolds appear in sub-Riemannian geometry (cf. \cite{bellaiche1996tangent}), with $A_1\subset TM$ a distribution such that for all $k>0$, iterated Lie brackets of sections of $A_1$ of depth $\le k$ span a subbundle $A_k\subset TM$, and moreover there is some $r<\infty$ such that $A_r=TM$ (i.e. $A_1$ is an equiregular bracket-generating distribution). Filtrations also appear in the study of hypoelliptic differential operators, where the first-order differential operators associated to vector fields in $\Gamma(A_k)$ are typically assigned a `filtered order' $k$ (instead of $1$), making filtrations a natural setting for studying certain hypoelliptic (but not elliptic) differential operators given as a sum of terms of varying orders.

A Lie filtration $A_\bullet$ of $TM$ determines a bundle of simply connected nilpotent Lie groups over $M$ called the \emph{osculating bundle} of the filtration. In the context of the study of hypoelliptic operators, the appropriate (co-)symbol of the operator is a smooth family of left-invariant differential operators on the fibers of the osculating bundle, and the representation theory of the fibers plays a role in elucidating properties of the operators; an important example is the celebrated Rockland condition, which relates the Fredholm property of an operator to a non-commutative analogue of ellipticity for its symbol. Another important construction in this context is the filtered analogue of Connes' tangent groupoid, a manifold that smoothly interpolates between $M\times M$ (the home of the Schwartz kernel of an operator) and the osculating bundle (the home of its symbol). In the unfiltered setting, Debord-Skandalis \cite{debord2014adiabatic} used the tangent groupoid to give a conceptual approach to pseudo-differential calculus; likewise the filtered analogue of the tangent groupoid plays an important role in the study of non-standard pseudo-differential calculi adapted to the study of hypoelliptic operators in various contexts, cf. the work of van Erp-Yuncken \cite{van2019groupoid}. A sample of other works in this direction include \cite{HigsonEulerLike, choi2019privileged, choi2019tangent, androulidakis2022pseudodifferential, mohsen2026microlocal}. In the special case that the filtration comes from a contact structure, there was an earlier significant literature on the corresponding pseudo-differential calculus, known as the Heisenberg calculus, cf. \cite{van2019groupoid} for further discussion and references.

In an earlier paper \cite{LMsingularlie}, Meinrenken and the author showed that a filtration of $M$ canonically determines a weighting of $M\times M$ along the diagonal $M$. There are weighted variants of several useful constructions in differential and algebraic geometry: weighted normal bundle, weighted deformation to the normal cone, weighted projective spaces and blow-ups. In this context, the weighted normal bundle $\nu_w(M\times M,M)$ recovers the osculating bundle, and the weighted deformation to the normal cone $\D_w(M\times M,M)$ recovers the filtered analogue of Connes' tangent groupoid mentioned above. The setting of weightings of manifolds provides a conceptual framework in which these and other constructions find a natural home. For instance, manifolds equipped with weightings and weighted morphisms form a category, and the weighted normal bundle and weighted deformation to the normal cone constructions are functorial.

Lie filtrations can be defined more generally for a Lie algebroid $A$ over a smooth manifold $M$. This is a natural direction of generalization for applications to the study of hypoelliptic operators since Lie algebroids can be used for instance to study operators on non-compact manifolds, compactified to manifolds with fibered boundary or corners, cf. \cite{melrose1993atiyah}. Lie filtrations of Lie algebroids are also discussed in \cite{van2019groupoid}. In this case the pair groupoid $\tn{Pair}(M)=M\times M$ above is replaced by a Lie groupoid $G\rightrightarrows M$ with Lie algebroid $A$. It then becomes natural to study weightings of a Lie groupoid $G$ along its submanifold of units $M$. 

Let $G\rightrightarrows M$ be a Lie groupoid with Lie algebroid $A$. A \emph{multiplicative weighting} of $G$ along $M$ is a weighting of $G$ along $M$ such that the groupoid multiplication $m\colon G\times_M G\rightarrow G$ is a weighted morphism with respect to a naturally induced weighting on the space of composable arrows $G\times_M G$. There is a one-one correspondence between multiplicative weightings of $G$ along $M$ and Lie filtrations of $A$. This was studied for $G=\tn{Pair}(M)$ in \cite{LMsingularlie} and generalized to arbitrary Lie groupoids by Dan Hudson in his PhD thesis \cite[Theorem 5.28]{hudson2024multiplicative} (see also discussion of related results by Meinrenken in \cite{meinrenken2026introduction}); in fact Hudson obtained a more general result in the setting of a Lie groupoid with a multiplicative weighting along a wide Lie subgroupoid. We give a different, self-contained proof of this result of Hudson, in the special case of a weighting of $G$ along $M$. This is the case with a direct relationship to the deformation spaces studied by van Erp-Yuncken \cite{van2019groupoid}, the filtered analogue of Connes' tangent groupoid important in the study of hypoelliptic operators. We also add a further one-one correspondence with dg weightings of the NQ manifold $A[1]$:
\begin{theorem}
\label{t:three-way-equiv}
Let $G\rightrightarrows M$ be a Lie groupoid with Lie algebroid $A$. There is a canonical one-one correspondence between the following:
\begin{enumerate}[1.]
\item multiplicative weightings $\I_\bullet$ of $G$ along $M$,
\item Lie filtrations $A_\bullet$ of $A$, and
\item dg weightings $\I^{\Omega_A}_\bullet$ of the NQ manifold $A[1]$ along $M$.
\end{enumerate}
\end{theorem}
In brief the equivalence $2\leftrightarrow 3$ is the duality between structures on $A$ compatible with the Lie bracket and anchor, and structures on the Chevalley-Eilenberg-de Rham complex $(\Omega_A,\d_A)$ compatible with the differential. The construction $1\rightarrow 2$ obtains $A_k$ as the $k$-th filtered piece $\nu(G,M)_{(k)}$ for the filtration of the normal bundle $\nu(G,M)=A$ associated to the weighting. The construction $2\rightarrow 1$ obtains a weighting $\I_\bullet$ by first using the Lie filtration $A_\bullet$ to introduce a filtered notion of order (called `$A_\bullet$-order') for left-invariant differential operators $P^L$ on $G$, and then setting $\I_k$ to be smooth functions $f$ such that $P^L f|_M=0$ for all $P^L$ with $A_\bullet$-order less than $k$. Along the way to Theorem \ref{t:three-way-equiv} we prove that the weighted deformation space functor $\D_w$, from the category of weighted manifolds to the category of $\bR^\times$-manifolds$/\bR$, is fully faithful (Theorem \ref{t:full-faithful}); this supports the point of view that weightings isolate the minimal structure relevant to the study of weighted deformation spaces.

The contents of the article are as follows. Section 2 discusses the character spectrum of an algebra over $\bR$, drawing inspiration from \cite{HigsonEulerLike}. We observe that the spectrum construction defines a functor from algebras to differential spaces (in the sense of Sikorski), provide a criterion for the spectrum to be a manifold (Proposition \ref{p:char-spec}), and generalize to sheaves of algebras. Section 3 recalls the definition of weightings of manifolds and describes their basic properties from \cite{LMweightings}. Sections 4 and 5 introduce the weighted normal bundle and weighted deformation to the normal cone, using the character spectrum to give a coordinate-free approach as in \cite{HigsonEulerLike, LMweightings}, although our presentation here puts a little more emphasis on the functorial properties of the spectrum construction for sheaves of algebras. Section 6 discusses weightings of fiber products of manifolds, as preparation for Section 7 where the definition of multiplicative weightings of Lie groupoids is recalled, and the weighted normal bundle/weighted deformation spaces of a Lie groupoid with a multiplicative weighting along its units are shown to themselves be Lie groupoids. Section 8 proves the correspondence $2\leftrightarrow 3$ of Theorem \ref{t:three-way-equiv}; Sections 9 and 10 establish $2\rightarrow 1$ and $1\rightarrow 2$ respectively.

\bigskip
\noindent \textbf{Acknowledgements.} I thank Eckhard Meinrenken, Dan Hudson, Nigel Higson, and Ahmad Sadegh for helpful conversations. Much of the material in this note originated with a mini-course taught at KU Leuven by the author in June 2026. I thank the organizers Marco Zambon, Miquel Cueca, and Zan Grad for the opportunity, and the participants for feedback and discussions. The mini-course also featured related material on weighted Euler-like vector fields and normal form theorems, but I decided to omit these here, as there are already many references for the latter material in the literature, cf. \cite{bursztyn2019splitting, HigsonEulerLike, bischoff2020deformation, meinrenken2021euler}. I thank Boris Zupancic for pointing out the references \cite{brais2025streamlining, lapointe2026weighted}. ChatGPT was used to proof-read the final draft and caught some minor errors and omissions; the author is responsible for any remaining errors.

\section{Spectrum in differential geometry}
In this section we review the definition of the character spectrum of an algebra, and more generally sheaves of algebras. The character spectrum has a natural differential space structure (in the sense of Sikorski). We describe a criterion for the spectrum to be a manifold, which is a variation of an analogous criterion that appeared in \cite[Lemma 2.4]{HigsonEulerLike}.

Let $A$ be a commutative algebra over $\bR$. 
\begin{definition}
A \emph{character} (or $\bR$-\emph{point}) $p$ of $A$ is a homomorphism $p\colon A\rightarrow \bR$.
\end{definition}
Consider the set
\[ \Spec(A):=\Hom_{\tn{alg}}(A,\bR) \]
of all characters of $A$. Note that each $a\in A$ determines a function
\[ a\colon \tn{Spec}(A)\rightarrow \bR, \qquad a(p)=p(a).\]
We equip $\tn{Spec}(A)$ with the topology generated by the sets $a^{-1}(J)$ where $a \in A$ and $J\subset \bR$ is open, i.e. the coarsest topology making the functions $a \in A$ continuous. If $p_1,p_2$ are distinct characters then by definition there is some $a \in A$ such that $p_1(a)\ne p_2(a)$, and in particular it follows that $\tn{Spec}(A)$ is Hausdorff.

We endow $S=\tn{Spec}(A)$ with a sheaf of functions $C^\infty_S$ defined as follows. Let $V\subset S$ be open. Then $C^\infty_S(V)$ is defined to be the collection of all functions $f\colon V\rightarrow \bR$ such that for each $p \in V$ there exists an open neighborhood $U$ of $p$, an integer $k\ge 1$, algebra elements $a_1,...,a_k\in A$, and a smooth $\bR$-valued function $g$ defined on a neighborhood of the image of $(a_1|_U,...,a_k|_U)\colon U\rightarrow \bR^k$, such that $f|_U=g(a_1|_U,...,a_k|_U)$. Then $C^\infty_S$ is a subsheaf of the sheaf of continuous $\bR$-valued functions on $S$ closed under post-composition with $C^\infty$-functions: for any $k\ge 1$, $U\subset S$ open, $f_1,...,f_k\in C^\infty_S(U)$, and any $g\in C^\infty(\bR^k)$, the function $g(f_1,...,f_k)\in C^\infty_S(U)$. Conversely a topological space $S$ equipped with a sub-sheaf $C^\infty_S$ of the sheaf of continuous $\bR$-valued functions such that 
\begin{itemize}
\item the topology on $S$ is the weakest such that the elements of $C^\infty_S$ are continuous, and 
\item $C^\infty_S$ is closed under post-composition with $C^\infty$-functions,
\end{itemize}
is known as a \emph{differential space} (in the sense of Sikorski), cf. \cite[Definition 2.48]{lerman2024differential}. A morphism of differential spaces is a continuous map $f\colon S_1\rightarrow S_2$ such that $f^*C^\infty_{S_2}\subset C^\infty_{S_1}$. The obvious functor $M\mapsto (M,C^\infty_M)$ from the category of manifolds to the category of differential spaces is fully faithful.

\begin{definition}
The \emph{character spectrum} $(S=\Spec(A),C^\infty_S)$ of $A$ is the differential space whose underlying set 
\[ S=\tn{Spec}(A)=\Hom_{\tn{alg}}(A,\bR),\] 
is the set of all characters of $A$, equipped with the topology and the sheaf $C^\infty_S$ defined above. We will often call $\Spec(A)$ simply the \emph{spectrum} of $A$ for short, as we will not consider any other notion of spectrum. Local sections of $C^\infty_S$ are referred to as \emph{smooth functions} and $C^\infty_S$ is referred to as the \emph{sheaf of smooth functions}.
\end{definition}
\begin{example}
\label{ex:milnor-exercise}
For any smooth manifold $M$, the spectrum $\tn{Spec}(C^\infty(M))=\Hom_{\tn{alg}}(C^\infty(M),\bR)\simeq M$, the isomorphism sends $p\in M$ to the algebra homomorphism $f\in C^\infty(M)\mapsto f(p)\in \bR$. This fact is sometimes referred to as `Milnor's exercise', cf. \cite[Corollary 35.9]{kolar2013natural}. As an example application, if $U$ is an open subset of $\bR^n$, then an algebra homomorphism $u\colon C^\infty(U)\rightarrow \bR$ is uniquely determined by its values on the $n$ coordinate functions $(u(x^1),...,u(x^n))\in \bR^n$, and moreover the elements of $\bR^n$ that occur in this way are precisely those that lie in the open subset $U$.
\end{example}

\begin{proposition} 
$\Spec(-)=\Hom_{\tn{alg}}(-,\bR)$ extends to a contravariant functor from the category of $\bR$-algebras to the category of differential spaces.
\end{proposition}
\begin{proof}
A map of algebras $f\colon A\rightarrow B$ induces a map of sets
\[ \ti{f}\colon \Spec(B)=S_B\rightarrow \Spec(A)=S_A.\] 
If $a \in A$ and $U=a^{-1}(J)\subset \Spec(A)$ is an open subset then $\ti{f}^{-1}(U)=(a\circ \ti{f})^{-1}(U)=(f(a))^{-1}(U)\subset \Spec(B)$ is an open subset, hence $\ti{f}$ is continuous. If $g(a_1|_U,...,a_k|_U)\in C^\infty_{S_A}(U)$ then
\[ \ti{f}^*g(a_1|_U,...,a_k|_U)=g(f(a_1)|_{\ti{f}^{-1}(U)},...,f(a_k)|_{\ti{f}^{-1}(U)})\in C^\infty_{S_B}(\ti{f}^{-1}(U)) \]
which verifies that pullback $\ti{f}^*$ maps elements of $C^\infty_{S_A}$ to elements of $C^\infty_{S_B}$, and hence $\ti{f}$ is a morphism of differential spaces.
\end{proof}

We will want a criterion describing when the topological space $S=\tn{Spec}(A)$ can be equipped with the structure of a smooth manifold such that $C^\infty_S$ is its sheaf of smooth functions. 

\begin{definition}
\label{d:local-coord}
Let $U\subset S=\Spec(A)$ be an open subset and $a^1,...,a^n\in C^\infty_S(U)$. We will say that $(a^1,...,a^n)$ are \emph{local coordinates on $U$} if
\begin{itemize}
\item $a=(a^1,...,a^n)\colon U\rightarrow \bR^n$ is a bijection to an open subset $a(U)\subset \bR^n$.
\item $C^\infty_S(U)$ is \emph{smoothly generated} by $a^1,...,a^n$ in the sense that $C^\infty_S(U)=\{\varphi(a^1,...,a^n)\mid \varphi \in C^\infty(a(U))\}$.
\end{itemize}
\end{definition}
\begin{proposition}
\label{p:homeo-smooth-gen}
Let $U\subset S=\Spec(A)$ admit local coordinates $a=(a^1,...,a^n)$. Then $a\colon U\rightarrow a(U)$ is a homeomorphism. Moreover for any open $V\subset U$, $C^\infty_S(V)$ is smoothly generated by $a^1|_V,...,a^n|_V$.
\end{proposition}
\begin{proof}
$a$ is continuous because its components $a^i \in C^\infty_S(U)$, and $C^\infty_S$ is a subsheaf of the sheaf of continuous functions on the spectrum. We must show that $a$ is an open map. Subsets $b^{-1}(J)\cap U=(b|_U)^{-1}(J)$ for $b \in A$ and opens $J\subset \bR$, generate the topology of $U$. Since $a^1,...,a^n$ smoothly generate $C^\infty_S(U)$, there exists a smooth function $\varphi$ such that $b|_U=\varphi(a^1,...,a^n)=\varphi\circ a$. Then
\[ a((b|_U)^{-1}(J))=a\circ a^{-1}\circ \varphi^{-1}(J)=\varphi^{-1}(J) \]
which is open since $\varphi$ is continuous.

To prove the second claim, let $f \in C^\infty_S(V)$. By the first part of the proof $a(V)$ is an open subset of $\bR^n$. By the definition of $C^\infty_S$, there is an open cover $V=\cup V_i$, and for each $V_i$, elements $b_{ij} \in A$ for $j=1,...,j_i$, and smooth functions $\varphi_i$ defined on an open subset of $\bR^{j_i}$, such that $f|_{V_i}=\varphi_i(b_{i1}|_{V_i},...,b_{ij_i}|_{V_i})$. By the smooth generation hypothesis, there exist $\psi_{ij}\in C^\infty(a(U))$, $j=1,...,j_i$, such that $b_{ij}|_U=\psi_{ij}(a^1,...,a^n)$. Thus 
\[ f|_{V_i}=\varphi_i(\psi_{i1}(a^1,...,a^n)|_{V_i},...,\psi_{ij_i}(a^1,...,a^n)|_{V_i})=F_i\circ a|_{V_i},\]
where $F_i=\varphi_i(\psi_{i1},...,\psi_{ij_i})$. By composing with $(a|_{V_i\cap V_j})^{-1}$ we deduce that
\begin{equation} 
\label{e:FiFj}
F_i\circ a|_{V_i\cap V_j}=f|_{V_i\cap V_j}=F_j\circ a|_{V_i\cap V_j} \quad \Rightarrow \quad F_i|_{a(V_i\cap V_j)}=F_j|_{a(V_i\cap V_j)}.
\end{equation}
By the first part of the proof, $a(V_i)$, $a(V_i\cap V_j)$ are open subsets of $a(V)$, hence \eqref{e:FiFj} shows that the functions $F_i$ glue together to a smooth function $F$ on $a(V)$, and then $f=F\circ a|_V$. This proves that $a^1|_V,...,a^n|_V$ smoothly generate $C^\infty_S(V)$.
\end{proof}

The next proposition is the desired criterion for $(S=\Spec(A),C^\infty_S)$ to be manifold. The neat observation is that it is not necessary to check smoothness of transition functions, as this is automatic. The criterion is a variant of \cite[Lemma 2.4]{HigsonEulerLike}.

\begin{proposition}
\label{p:char-spec}
If $S=\Spec(A)$ can be covered by countably many open subsets that admit local coordinates, then $S$ admits the structure of a smooth manifold with sheaf of smooth functions $C^\infty_S$.
\end{proposition}
\begin{proof}
The hypothesis, the first condition in Definition \ref{d:local-coord}, and the first part of Proposition \ref{p:homeo-smooth-gen}, together say that $S$ is covered with open subsets, each of which can be identified with an open subset of $\bR^n$; these are the candidate manifold charts. The second condition in Definition \ref{d:local-coord} and the second part of Proposition \ref{p:homeo-smooth-gen} together imply that the transition functions are smooth. Indeed 
If $a=(a^1,...,a^n)|_U\rightarrow \bR^n$ and $\tilde{a}=(\tilde{a}^1,...,\tilde{a}^n)|_{\tilde{U}}\colon \tilde{U}\rightarrow \bR^n$ are two such charts, then by the second part of Proposition \ref{p:homeo-smooth-gen}, there are smooth functions $\phi_i$ such that $a_i|_{U\cap \tilde{U}}=\phi_i\circ (\tilde{a}|_{U\cap \tilde{U}})$. Then $a|_{U\cap \tilde{U}}=\phi\circ (\tilde{a}|_{U\cap \tilde{U}})$ where $\phi=(\phi_1,...,\phi_n)$ is smooth. Pre-compose with $(\tilde{a}|_{U\cap \tilde{U}})^{-1}$ to conclude that $a|_{U\cap \tilde{U}}\circ (\tilde{a}|_{U\cap \tilde{U}})^{-1}=\phi\colon \tilde{a}(U\cap \tilde{U})\rightarrow a(U\cap \tilde{U})$ is smooth.
\end{proof}

Suppose one is given a sheaf $\A$ of $C^\infty_M$-algebras over an auxiliary manifold $M$. Let $\S=\Spec(\A)$ denote the composite functor, from the category of open subsets of $M$ to the category of differential spaces. Since $\A,\Spec$ are both contravariant, $\S$ is covariant. Since $\Spec(C^\infty(U))=U$, functoriality of $\Spec$ yields compatible morphisms
\[ \pi_U\colon \S(U)\rightarrow U \]
for every open $U\subset M$, in other words, a natural transformation $\pi\colon \S\rightarrow \tn{Id}$.

\begin{lemma}
Let $M$ be a manifold, $\A$ a sheaf of $C^\infty_M$-algebras, $\S=\Spec(\A)$, and let $\pi\colon \S\rightarrow \tn{Id}$ be the corresponding natural transformation. Let $j\colon U\hookrightarrow M$ be an open inclusion. The morphism $\S(j)\colon \S(U)\hookrightarrow \S(M)$ is an open embedding that identifies $\S(U)$ with the open subset $\pi_M^{-1}(U)$, and $\pi_U=\pi_M|_{\S(U)}$.
\end{lemma}
\begin{proof}
By naturality $j\circ \pi_U=\pi_M\circ \S(j)$ and hence $\pi_M^{-1}\circ j=\S(j)\circ \pi_U^{-1}$. It follows that $\S(j)(\S(U))=\S(j)(\pi_U^{-1}(U))=\pi_M^{-1}(j(U))=\pi_M^{-1}(U)$ so the image of $\S(j)$ is $\pi_M^{-1}(U)$, an open set. We will argue that $\S(j)$ is injective. Let $p\ne p'\in \S(U)$. Then there exists $a_U \in \A(U)$ such that $a_U(p)\ne a_U(p')$. Suppose by contradiction that $\S(j)(p)=\S(j)(p')$. By naturality $\pi_U(p)=\pi_M(\S(j)(p))=\pi_M(\S(j)(p'))=\pi_U(p')=:m$. Let $\chi\in C^\infty(M)$ be a bump function such that $\chi(m)=1$ and $\supp(\chi)\subset U$. Let $V$ be an open subset of $M$ such that $\ol{V}\cap \supp(\chi)=\emptyset$ and $U\cup V=M$. We claim that there is a unique section $a \in \A(M)$ such that $a|_U=j^*a=\chi a_U$ (where the product here is defined using the $C^\infty(U)$-algebra structure of $\A(U)$) and $a|_V=0$; indeed $\chi a_U|_{U\cap V}=0$ since $\ol{V}\cap \supp(\chi)=\emptyset$, so the existence and uniqueness of $a$ follows from the sheaf property of $\A$. Then 
\[ a_U(p)=\chi(m)a_U(p)=(j^*a)(p)=a(\S(j)(p))=a(\S(j)(p'))=a_U(p') \] 
contradicting $a_U(p)\ne a_U(p')$. This proves that $\S(j)$ is injective. The relation $\pi_U=\pi_M|_{\S(U)}$ then follows by naturality of $\pi$.
\end{proof}

\begin{corollary}
\label{c:spec-manifold}
Let $M$ be a manifold, $\A$ a sheaf of $C^\infty_M$-algebras, $\S=\Spec(\A)$, and let $\pi\colon \S\rightarrow \tn{Id}$ be the corresponding natural transformation. Let $S=\S(M)$. Let $\{U_i\}$ be a countable open cover of $M$ such that $\S(U_i)$ admits local coordinates. Then $S$ is a manifold with sheaf of smooth functions $C^\infty_S$, obtained by gluing together the smooth manifolds $\{\S(U_i)\}$ along the open subsets $\{\S(U_i\cap U_j)\}$, and equipped with a smooth map $\pi_M\colon S\rightarrow M$.
\end{corollary}

We end this section with a collection of examples including higher order tangent bundles of a manifold.
\begin{example}
\label{e:higher-tgt}
Let $(R,\mf{m})$ be a commutative local $\bR=R/\mf{m}$-algebra which is $d$-dimensional ($d<\infty$) as an $\bR$-vector space, and let $\varepsilon\colon R\rightarrow R/\mf{m}=\bR$ denote the quotient homomorphism. Let $S^R=\Spec^R(A)=\Hom_{\tn{alg}}(A,R)$ equipped with the weakest topology such that for all $a\in A$ the function $a^R\colon \Spec^R(A)\rightarrow R\simeq \bR^d$ is continuous, where $a^R(p)=p(a)$. Imitating our earlier discussion of the sheaf $C^\infty_S$ on $S=\Spec(A)$ but replacing the smooth $\bR$-valued function $g$ defined on an open subset of $\bR^k$ (resp. $a_1,...,a_k$) with a smooth $\bR$-valued function $g$ defined on an open subset of $R^k\simeq \bR^{kd}$ (resp. $a_1^R,...,a_k^R$), we obtain a sheaf $C^\infty_{S^R}$ of $\bR$-valued functions on $S^R$. The sheaf $C^\infty_S$ is realized as the sub-sheaf of $C^\infty_{S^R}$ obtained by considering those smooth functions $g\colon R^k\rightarrow \bR$ that factor through $\varepsilon^{\times k}\colon R^k\rightarrow \bR^k$; in particular this makes $C^\infty_{S^R}$ into a sheaf of $C^\infty_S$-algebras. Specialize to the case $A=C^\infty(M)$ for a smooth manifold $M$, hence $S=\Spec(C^\infty(M))=M$, and $\A=C^\infty_{S^R}$ is a sheaf of $C^\infty_M$-algebras. If $U\subset M$ is the domain of a coordinate chart $\phi=(x^1,...,x^n)$, and $r_1,...,r_d$ is a basis for the dual vector space to $R$, then $\{x_j^i=r_j\circ (x^i)^R\mid 1\le i\le n,1\le j\le d\}$ are local coordinates on $\S(U)=\Spec(\A(U))$. By Corollary \ref{c:spec-manifold}, $M^R=\Spec^R(C^\infty(M))$ is a smooth manifold equipped with a smooth map $\pi\colon M^R\rightarrow M$. The algebra $R$ is known as a \emph{Weil algebra}, an example being the `dual numbers' $R=\bR[\epsilon]/(\epsilon^2)$ for which $M^R=TM$, cf. \cite[Section 35]{kolar2013natural} for the relation to higher order tangent bundles.
\end{example}

\section{Weightings of manifolds}
In this section we recall the notion of a `weighting' of a manifold $M$ introduced in \cite{meinrenken2021euler, LMweightings} and, under the name `quasi-homogeneous structure', in \cite{melrose-corners-notes, behr-quasi-homog}. The motivation for the definition comes from studying the geometry of $M$ near a submanifold $N$ in situations where one wants to coherently assign weights to the local coordinates in appropriate submanifold charts; to make this precise and obtain a global description, we use sheaves of ideals. For an equivalent formulation using jet bundles see \cite{LMweightings}.

\begin{definition}
Let $n\in \bZ_{\ge 0}$. A (\emph{length} $n$) \emph{weight sequence} $w=(w_1,...,w_n)$ is an $n$-tuple of non-negative integers $w_1\le w_2\le \cdots \le w_n$. The components $w_i$ are \emph{weights}. The corresponding \emph{weighted scalar multiplication} on $\bR^n$ is the $(\bR,\cdot)$-monoid action given by $\lambda\colon x=(x^1,...,x^n)\mapsto \lambda^wx=(\lambda^{w_1}x^1,...,\lambda^{w_n}x^n)$.
\end{definition}
An open subset $U\subset \bR^n$ and a weight sequence $w=(w_1,...,w_n)$ together determine $\{I_k\}$, a decreasing sequence of ideals
\[ C^\infty(U)=I_0\supset I_1\supset I_2\supset \cdots \]
where
\begin{equation} 
\label{e:wt-loc-form}
I_k=(x^s\mid s\cdot w\ge k), \quad x^s=(x^1)^{s_1}\cdots (x^n)^{s_n}, \quad s\cdot w=\sum_{j=1}^n s_jw_j,
\end{equation}
and $(x^s\mid s\cdot w\ge k)$ denotes the ideal generated by the monomials $x^s$ for $s \in \bZ^n_{\ge 0}$ satisfying $s\cdot w\ge k$. The ideals $I_k$ can be characterized in terms of the weighted scalar multiplication.
\begin{lemma}
\label{l:Ik-dyn}
Let $w=(w_1,...,w_n)$ be a weight sequence. Let $U\subset \bR^n$ be an open subset invariant under weighted scalar multiplication. Then $I_k\subset C^\infty(U)$ consists of those smooth functions $f$ such that $\lim_{\lambda\rightarrow 0}\lambda^{-k}f(\lambda^wx)$ exists for all $x\in U$.
\end{lemma}
\begin{proof}
If $f\in I_k$ it is immediate that the limit in the statement exists. Conversely suppose $f\in C^\infty(U)$ has the property that $\lim_{\lambda\rightarrow 0}\lambda^{-k}f(\lambda^wx)$ exists for all $x\in U$. Let $w_1,...,w_p$ be the weights that equal $0$, hence $0<w_{p+1}\le \cdots \le w_n$ are the positive weights. Let $U_0=U\cap \{x^{p+1}=\cdots=x^n=0\}$ be the subspace of $U$ where the coordinates with positive weight vanish. Consider the Taylor expansion of $f$ in the variables $x^{p+1},...,x^n$ around $U_0$ up to order $k$; the remainder term in the Taylor expansion belongs to $I_{U_0}^k$ where $I_{U_0}$ is the vanishing ideal of $U_0$. One has $I_{U_0}^k\subset I_k$, hence the remainder term belongs to $I_k$. On the other hand, the $k$-th order Taylor approximation is a finite sum of terms, each of which is a smooth function of $x^1,...,x^p$ times a monomial $x^s$ in $x^{p+1},...,x^n$. Existence of the limit implies that the only monomials which can occur are those with $s\cdot w\ge k$, and hence the $k$-th order Taylor approximation also belongs to $I_k$.
\end{proof}

\begin{definition}
A \emph{weighting} of $M$ is a sequence of ideal sheaves
\[ C^\infty_M=\I_0\supset \I_1\supset \I_2\supset\cdots \]
such that for every $p \in M$, there exists a coordinate chart $(U,\phi=(x^1,...,x^n))$ around $p$ and a weight sequence $w=(w_1,...,w_n)$, such that in $\phi$-coordinates, $\{\I_k(U)\}$ is of the form \eqref{e:wt-loc-form}. We refer to such a coordinate chart as a \emph{weighted coordinate chart}. The maximal weight $r$ across all weighted coordinate charts is assumed to be finite and is called the \emph{order} of the weighting. The pair consisting of the manifold $M$, together with a weighting of $M$, is called a \emph{weighted manifold}. A morphism of weighted manifolds from $(M,\I_\bullet)$ to $(M',\I'_\bullet)$ is a smooth map $f\colon M\rightarrow M'$ such that $f^*\I_k'\subset \I_k$ for all $k$.
\end{definition}
We mention a number of basic properties of weightings; for details see \cite{LMweightings}.
\bigskip

\noindent $\blacktriangleright$ \emph{The submanifold determined by the weighting.} The ideal $\I_1$ is generated locally by coordinates that have strictly positive weight in some weighted coordinate chart. It follows that $\I_1=\I_N$ is the vanishing ideal of a smooth topologically-closed embedded submanifold $N\subset M$, and those weighted coordinate charts that intersect $N$ non-trivially are automatically submanifold charts for $N$. For this reason one often includes $N$ in the description of a weighting, for instance referring to a \emph{weighting of $M$ along $N$}. We think of $\I_k$ as the sheaf of smooth functions that \emph{vanish to weighted order $k$ along $N$}. If all weights are either $0$ or $1$ (in other words, $r=1$) we call this the \emph{standard weighting} of $M$ along $N$, and $\I_k=\I_N^k$. Increasing some of the non-zero weights from $1$ to $2,3,4,...$ changes the decreasing sequence of ideals, making $\I_k$ \emph{larger} than $\I_N^k$: certain coordinates are considered to have higher weighted vanishing order on $N$, and so survive further along the decreasing sequence of ideals, making those ideals larger than they were in the case of the standard weighting of $M$ along $N$. Since a weighted morphism $f\colon M\rightarrow M'$ satisfies $f^*\I_1'\subset \I_1$, it follows that $f(N)\subset N'$, that is, $f\colon (M,N)\rightarrow (M',N')$ is, in particular, a map of pairs; the condition $f^*\I_k'\subset \I_k$ for $k\ge 2$ imposes further, higher order conditions on the behavior of $f$ near $N$.
\begin{remark}
\label{r:nested-submanifolds}
One simple way to specify a weighting is to give a sequence of nested submanifolds
\[ N=N_0\subset N_1\subset N_2\subset \cdots \subset N_{r-1}\subset N_r=M \]
set $\J_j=\I_{N_{j-1}}$ and $\I_k=\sum_i \sum_{j_1+...+j_i=k}\J_{j_1}\cdots\J_{j_i}$. All weightings are of this form for some sequence of nested submanifolds, but the choice of submanifolds is highly non-unique. 
\end{remark}
\bigskip

\noindent $\blacktriangleright$ \emph{Multiplicative properties of the ideals.} An immediate consequence of the local coordinate description is that
\[ \I_k\cdot \I_l\subset \I_{k+l}.\]
The order $r$ of the weighting is the smallest integer such that $\I_{r+1}\subset \I_N^2$, meaning in particular that in any weighted submanifold chart for $N$, the highest possible weight of a coordinate function is $r$. One can show \cite{LMweightings} that for $k\ge 2$,
\begin{equation} 
\label{e:higher-ideals-determined}
\I_k\cap \I_N^2=\sum_{0<j<k} \I_j\I_{k-j}.
\end{equation}
This, together with the containment $\I_{r+1}\subset \I_N^2$, imply that the weighting is already determined by the ideals $\I_1\supset \I_2\supset\cdots \supset \I_r$.
\bigskip

\noindent $\blacktriangleright$ \emph{First-order data along $N$.} As one might expect, a weighting determines `first-order' directional information along $N$ in the form of a filtration of $T_NM=TM|_N$ by subbundles
\[ T_NM_{(0)}=TN\subset T_NM_{(1)}\subset \cdots \subset T_NM_{(r)}=T_NM,\]
where $T_NM_{(k)}\subset T_NM$ consists of the vectors such that the corresponding directional derivative annihilates $\I_{k+1}$. If $(U,\phi=(x^1,...,x^n))$ is a weighted coordinate chart, the coordinate vector fields $\{\partial/\partial x^i\mid w_i\le k\}$ restrict to a local frame for $T_NM_{(k)}$ over $U\cap N$. Taking quotients by $TN$, we obtain a filtration
\begin{equation}  
\label{e:normal-bun-filtration}
\nu(M,N)_{(0)}=0\subset \nu(M,N)_{(1)}\subset \cdots \subset \nu(M,N)_{(r)}=\nu(M,N)
\end{equation}
of the normal bundle $\nu(M,N)=T_NM/TN$. The associated graded vector bundles are denoted 
\[  \gr(T_NM)=\bigoplus_{k=0}^r \gr^k(T_NM),\quad \gr(\nu(M,N))=\bigoplus_{k=0}^r \gr^k(\nu(M,N)). \]
There is a dual, decreasing filtration of $T^*_NM$, with $T^*_NM_{(k)}$ spanned by differentials along $N$ of elements of $\I_k$; equivalently the sheaf of sections of $T^*_NM_{(k)}$ is the $C^\infty_M/\I_1=C^\infty_M/\I_N=C^\infty_N$-module $\I_k/(\I_k\cap \I_N^2)$. The associated graded $\gr(T^*_NM)$ has $k$-th graded piece $\gr^k(T^*_NM)=T^*_NM_{(k)}/T^*_NM_{(k+1)}$, whose sheaf of sections is $\I_k/(\I_{k+1}+\I_k\cap \I_N^2)$. A weighted morphism $f\colon (M,N)\rightarrow (M',N')$ induces morphisms $\d_Nf_{(k)}\colon T_NM_{(k)}\rightarrow T_{N'}M'_{(k)}$. Since $\d_Nf(TN)\subset TN'$, $\d_Nf_{(k)}$ descends to  a morphism $\nu(f)_{(k)}\colon \nu(M,N)_{(k)}\rightarrow \nu(M',N')_{(k)}$. The corresponding morphisms of associated graded bundles are denoted $\gr(\d_Nf)$, $\gr(\nu(f))$ respectively.
\bigskip

\noindent $\blacktriangleright$ For non-standard weightings, a natural question is: does the first-order data along $N$ described above determine the weighting? The answer is `yes' when the order of the weighting $r=2$ and `no' for $r>2$. We elaborate on this in the next three remarks.

\begin{remark}
The data of a weighting of order $r=2$ is equivalent to giving a non-trivial subbundle $F=\nu(M,N)_{(1)}\subset \nu(M,N)$: given a weighting of order $r=2$ set $F=\tn{ann}(\{\d f\mid f \in \I_2\})/TN$, and conversely given $F$ define $\I_2=\{f\in \I_N\mid \d f|_F=0\}$. This special case was studied in \cite{meinrenken2021euler}.
\end{remark}
\begin{remark}
For $r>2$, the filtration of $T_NM$ described above is \emph{not} sufficient data to uniquely determine the weighting. For example consider $\bR^2$ and the following order $r=3$ weightings of $N=\{(0,0)\}$: (a) $x$ is weight 1 and $y$ is weight 3, (b) $x$ is weight 1 and $\ti{y}=y-x^2$ is weight 3. These have
\[ \I_2=(y,xy,x^2),\: \I_3=(y,xy,x^3),... \quad \text{and} \quad \ti{\I}_2=(\ti{y},x\ti{y},x^2),\: \ti{\I}_3=(\ti{y},x\ti{y},x^3),... \]
and they are not the same weighting because $\ti{y}=y-x^2\in \ti{\I}_3$ but $y-x^2\notin \I_3$. On the other hand they induce the same filtration $0\subset \tn{span}(\partial/\partial x)\subset \tn{span}(\partial/\partial x)\subset T_0\bR^2$.
\end{remark}
\begin{remark}
\label{r:jets}
In order to faithfully capture weightings then, higher order infinitesimal information along $N$ is needed. For instance, building on Remark \ref{r:nested-submanifolds}, to eliminate the non-uniqueness mentioned there, one could work with appropriate jets of submanifolds around $N$. In a similar vein, one can work with subbundles of the $r$-th order tangent bundle along $N$, an approach pursued in \cite{LMweightings}. In brief, a weighting of order $r$ determines (and is determined by) the subbundle $Q\subset T_rM|_N$ consisting of $r$-th order jets of curves $j^r\gamma$ such that for all $1\le k\le r$ and $f\in \I_k$, $j^r(f\circ \gamma)$ vanishes to order $k$.
\end{remark}

The next proposition is a criterion for a coordinate chart to be a weighted coordinate chart: if the coordinates have the correct weights (with multiplicities) to be a weighted coordinate chart, then near $N$ the coordinate chart is indeed a weighted coordinate chart. The result is essentially \cite[Proposition 1.2.8]{brais2025streamlining}; for completeness we outline a proof.
\begin{proposition}
\label{p:criterion-for-wt-chart}
Let $(M^n,N^p,\I_\bullet)$ be a weighted manifold. Let $(U,\phi=(x^1,...,x^n))$ be a weighted coordinate chart, where $x^j$ has weight $w_j$. Let $(U,\psi=(y^1,...,y^n))$ be another coordinate chart. If $y^j\in \I_{w_j}$ for $j=1,...,n$ then there is an open neighborhood $V\subset U$ of $U\cap N$ such that $\psi|_V$ is a weighted coordinate chart.
\end{proposition}
\begin{proof}
Without loss of generality we can assume $U=M$ and $U\cap N=N$. The two coordinate systems determine two weightings on $U$ having the same order $r$ and weight sequence $w$:
\begin{equation} 
\label{e:wt-loc-form2}
\I_k=(x^s\mid s\cdot w\ge k), \qquad \I_k'=(y^s\mid s\cdot w\ge k),
\end{equation}
where $x=(x^1,...,x^n)$, $y=(y^1,...,y^n)$, $s=(s_1,...,s_n)$, $w=(w_1,...,w_n)$. Since all higher ideals are determined (see \eqref{e:higher-ideals-determined}), the claim amounts to proving that $\I_k=\I_k'$ for $k=1,...,r$. The condition $y^j\in \I_{w_j}$, and the multiplicative property $\I_l\I_m\subset \I_{l+m}$, together imply that the generators $y^s$, $s\cdot w\ge k$ of $\I_k'$ belong to $\I_k$, hence $\I_k'\subset \I_k$ for all $k$.

The hypothesis implies that $y^{p+1},...,y^n\in \I_1=\I_N$. After shrinking $U$ to a smaller open neighborhood $V$ of $N$ if necessary, we may assume that the common vanishing locus of $y^{p+1},...,y^n$ is exactly $N$, and hence $\psi=(y^1,...,y^n)$ is a submanifold chart for $N$ and $\I_1'=\I_N=\I_1$. Proceed by induction: let $1<k$ and suppose $\I_l=\I_l'$ for $l<k$. The quotient sheaf $\I_k/\I_{k-1}$ (resp. $\I_k'/\I_{k-1}'=\I_k'/\I_{k-1}$) is a $C^\infty_M/\I_1=C^\infty_M/\I_N=C^\infty_N=C^\infty_M/\I_1'$-module; it is locally free, with local frame given by $x^s+\I_{k-1}$ with $s\cdot w=k$ (resp. $y^s+\I_{k-1}'$ with $s\cdot w=k$). Since the weight sequence $w$ for $\I_\bullet$, $\I_\bullet'$ is the same, $\I_k/\I_{k-1}$, $\I_k'/\I_{k-1}'$ have the same rank. Since $\I_k'\subset \I_k$, there is a surjection
\[ \I_{k-1}'/\I_k'=\I_{k-1}/\I_k'\twoheadrightarrow \I_{k-1}/\I_k. \]
As observed above, the domain and codomain of this surjection are locally free of the same rank, and hence the morphism is an isomorphism. It follows that $\I_k'=\I_k$, completing the inductive step.
\end{proof}

\section{Weighted normal bundle}\label{s:wt-normal}
In this section, following \cite{LMweightings}, we describe a variant of the normal bundle to a submanifold that incorporates information about a weighting $(M,N,\I_\bullet)$. The result is a fiber bundle $\nu_w(M,N)$ over $N$, the weighted normal bundle.

Let $(M^n,N,\I_\bullet)$ be a weighted manifold. Let
\[ \gr_{N,w}(C^\infty_M)=\bigoplus_{k\ge 0} \I_k/\I_{k+1} \]
be the associated graded, a sheaf of $C^\infty_M$-algebras, indeed even a sheaf of $C^\infty_M/\I_N=C^\infty_N$-algebras since $\I_N=\I_1$ acts trivially. The $\bZ_{\ge 0}$-grading determines an action $\kappa_{\bullet}^*\colon \bR\times \gr_{N,w}(C^\infty_M)\rightarrow \gr_{N,w}(C^\infty_M)$ of the monoid $(\bR,\cdot)$, where for $t \in \bR$, $\kappa_t^*$ acts on the summand $\gr_{N,w}^k(C^\infty_M)=\I_k/\I_{k+1}$ as scalar multiplication by $t^k$.
\begin{definition}
Let $(M,N,\I_\bullet)$ be a weighted manifold. The \emph{weighted normal bundle} to $N$ in $M$ is the character spectrum
\[ \nu_w(M,N)=\Spec(\gr_{N,w}(C^\infty_M)) \]
of the associated graded algebra. The $C^\infty_N$-algebra structure determines a morphism of differential spaces $\pi\colon \nu_w(M,N)\rightarrow N$, the quotient map $\gr_{N,w}(C^\infty_M)\rightarrow \gr_{N,w}^0(C^\infty_M)=\I_0/\I_1=C^\infty_N$ determines an inclusion $\iota\colon N\hookrightarrow \nu_w(M,N)$, and the $\bZ_{\ge 0}$-grading determines an action (by morphisms of differential spaces) $\kappa_\bullet\colon \bR\times \nu_w(M,N)\rightarrow \nu_w(M,N)$ of the monoid $(\bR,\cdot)$ such that $\kappa_0=\iota\circ \pi$.
\end{definition}

\begin{definition}
\label{d:homog-approx}
Let $(M,N,\I_\bullet)$ be a weighted manifold and let $f \in \I_k$. The \emph{weighted order $k$ homogeneous approximation to $f$} is the coset $f_{[k]}=f+\I_{k+1}\in \I_k/\I_{k+1}\subset \gr_{N,w}(C^\infty_M)$. The corresponding $\bR$-valued function on the spectrum $\nu_w(M,N)$ is denoted by the same symbol $f_{[k]}$, and is homogeneous of degree $k$ with respect to the action $\kappa$.
\end{definition}

\begin{theorem}[cf. Theorem 4.2 of \cite{LMweightings}]
\label{t:wt-norm}
The weighted normal bundle $\nu_w(M,N)$ is a manifold of dimension $n$.
\end{theorem}
\begin{proof}
Let $(U,\phi=(x^1,...,x^n))$ be a weighted coordinate chart, where $x^1,...,x^p \in \I_0$ are the coordinates that restrict to coordinates on $N$, i.e. the coordinates of weight $0$. The algebra $\gr_{N,w}(C^\infty_M(U))$ is generated, as a $C^\infty_N(U)$-algebra, by the weighted homogeneous approximations $\{x^i_{[w_i]}\mid i>p\}$. Recall that an element $u\in \Spec(\gr_{N,w}(C^\infty_M(U)))$ is a character $\gr_{N,w}(C^\infty_M(U))\rightarrow \bR$. The restriction of that character to the sub-algebra $\gr_{N,w}^0(C^\infty_M(U))=C^\infty_N(U)$ is a character $u_0=\pi(u)\in \Spec(C^\infty_N(U))$, which is the same thing as a point $u_0\in U\cap N$ (see Example \ref{ex:milnor-exercise}). Then $u$ is an extension of $u_0$ to a homomorphism $\gr_{N,w}(C^\infty_M(U))\rightarrow \bR$; that extension is uniquely determined by its value on the generators $\{x^i_{[w_i]}\mid i>p\}$, and conversely any choice of values on these generators determines an extension of $u_0$ to a character of $\gr_{N,w}(C^\infty_M(U))$. It follows that $(x^1_{[w_1]},...,x^n_{[w_n]})\colon \pi^{-1}(U)\rightarrow \bR^n$ is a bijection onto the open subset $(x^1,...,x^p)(U)\times \bR^{n-p}\subset \bR^p\times \bR^{n-p}=\bR^n$. By definition $C^\infty_{\nu_w(M,N)}(\pi^{-1}(U))$ is generated by smooth functions of elements of $\gr_{N,w}(C^\infty_M(U))$, and the latter is generated by $C^\infty_N(U)$ together with polynomials in the generators $\{x^i_{[w_i]}\mid i>p\}$; since polynomials are examples of smooth functions and compositions of smooth functions are smooth functions, $C^\infty_{\nu_w(M,N)}(\pi^{-1}(U))$ is smoothly generated by $\{x^i_{[w_i]}\mid 1\le i\le n\}$, which are therefore local coordinates for $\pi^{-1}(U)$. By Corollary \ref{c:spec-manifold}, $\nu_w(M,N)$ is a manifold.
\end{proof}
The weighted normal bundle is in particular a fiber bundle over $N$, with structure group the group of diffeomorphisms of $\bR^{\tn{codim}(N)}$ that commute with the $(\bR,\cdot)$ action (a finite-dimensional Lie group of diffeomorphisms having polynomial components of appropriate weighted degrees). The $(\bR,\cdot)$ action makes $\nu_w(M,N)$ into an example of a \emph{graded bundle} over $N$ in the sense of Grabowski-Rodkievicz \cite{grabowski2012graded}, and moreover $\nu_w$ extends to a functor from the category of weighted manifolds to the category of graded bundles.

By generalities on graded bundles, $\nu_w(M,N)$ is \emph{non-canonically} isomorphic, as a graded bundle, to its \emph{linearization} $\nu_w(M,N)_{\tn{lin}}:=\nu(\nu_w(M,N),N)$, a vector bundle over $N$ which is $\bZ_{\ge 0}$-graded according to weights of the $\bR^\times$-action on the fibers; in this case the linearization can be canonically identified with the associated graded vector bundle $\gr(\nu(M,N))$. Linearization extends to a conservative functor from the category of graded bundles to the category of graded vector bundles. Analogous to the familiar description $\nu(M,N)=TM|_N/TN$ of the normal bundle, it is possible \cite{LMweightings} to obtain the weighted normal bundle $\nu_w(M,N)$ as a quotient of the jet subbundle $Q$ from Remark \ref{r:jets}.

\section{Weighted deformation space}
In this section, following \cite{LMweightings}, we describe a variant of the deformation space (or deformation to the normal cone) construction that incorporates information about a weighting $(M,N,\I_\bullet)$. The result is a manifold $\D_w(M,N)$, the weighted deformation space, that interpolates between $M$ and $\nu_w(M,N)$ in the sense that $\D_w(M,N)$ is equipped with a canonical submersion $t$ to $\bR$ such that $t^{-1}(1)=M$ and $t^{-1}(0)=\nu_w(M,N)$.

Let $(M^n,N,\I_\bullet)$ be a weighted manifold. For convenience set $\I_k=\I_0$ for $k<0$.
\begin{definition}
The \emph{Rees algebra} associated to the filtration $\I_\bullet$ of $C^\infty_M$ is the sheaf of $C^\infty_M$-algebras
\[ \R_{N,w}(C^\infty_M)=\{\sum t^{-k}f_k\mid f_k\in \I_k\},\]
where $t$ is an auxiliary variable, a subalgebra of Laurent polynomials in $t$ with coefficients in $C^\infty_M$ ($\I_k\I_l\subset \I_{k+l}$ guarantees that this is a subalgebra). The \emph{weighted deformation to the normal cone} or \emph{weighted deformation space} is the differential space
\[ \D_w(M,N)=\Spec(\R_{N,w}(C^\infty_M)).\]
The inclusions $\bR[t],C^\infty_M\hookrightarrow \R_{N,w}(C^\infty_M)$ induce a morphism of differential spaces $(t,\hat{\pi})\colon \D_w(M,N)\rightarrow \bR\times M$. The algebra automorphisms $t\mapsto \lambda^{-1} t$, $\lambda \in \bR^\times$ determine an action (by morphisms of differential spaces) $\hat{\kappa}_\bullet\colon \bR^\times \times \D_w(M,N)\rightarrow \D_w(M,N)$ called the \emph{zoom action}.
\end{definition}
The fibers of the map $t\colon \D_w(M,N)\rightarrow \bR$ are not difficult to describe. For every $\tau \ne 0$ there is a homomorphism 
\[ \ev_\tau\colon \R_{N,w}(C^\infty_M)\rightarrow C^\infty_M, \qquad t^{-k}f_k\mapsto \tau^{-k}f_k, \] 
this induces an inclusion of differential spaces 
\[ j_\tau\colon M\hookrightarrow \D_w(M,N),\] 
whose image is the fiber of $\D_w(M,N)$ over $0\ne \tau \in \bR$. There is a homomorphism 
\[ \ev_0\colon \R_{N,w}(C^\infty_M)\rightarrow \gr_{N,w}(C^\infty_M), \qquad t^{-k}f_k\mapsto (f_k)_{[k]}=f_k+\I_{k+1}, \] 
this induces an inclusion of differential spaces
\[ j_0\colon \nu_w(M,N)\hookrightarrow \D_w(M,N),\] 
whose image is the fiber of $\D_w(M,N)$ over $\tau=0$. Thus, as sets,
\[ \D_w(M,N)=\nu_w(M,N)\sqcup (\bR^\times \times M).\]
The restriction of $\hat{\pi}$ (resp. $\hat{\kappa}$) to $\nu_w(M,N)$ is $\pi$ (resp. $\kappa$) as defined in Section \ref{s:wt-normal}. 
\begin{definition}
\label{d:dnc-wt-ext}
Let $(M,N,\I_\bullet)$ be a weighted manifold and let $f \in \I_k$. The \emph{weighted order $k$ homogeneous extension of $f$} is $\hat{f}_{[k]}=t^{-k}f\in \R_{N,w}(C^\infty_M)$. The corresponding $\bR$-valued function on the spectrum $\D_w(M,N)$ is denoted by the same symbol $\hat{f}_{[k]}$, and is homogeneous of degree $k$ with respect to the action $\hat{\kappa}$; in fact, as $\bR^\times \times M\subset \D_w(M,N)$ is open dense, it is the unique homogeneous degree $k$ smooth function satisfying $\hat{f}_{[k]}|_{t^{-1}(1)}=f$. Moreover $\hat{f}_{[k]}|_{t^{-1}(0)}=f_{[k]}$ is the weighted order $k$ homogeneous approximation to $f$ (Definition \ref{d:homog-approx}).
\end{definition}
\begin{theorem}[cf. Theorem 5.1 of \cite{LMweightings}]
The weighted deformation space $\D_w(M,N)$ is a manifold of dimension $n+1$.
\end{theorem}
\begin{proof}
Let $(U,\phi=(x^1,...,x^n))$ be a weighted coordinate chart where $x^1,...,x^p \in \I_0$ are the coordinates that restrict to coordinates on $N$, i.e. the coordinates of weight $0$. We claim that $\hat{\phi}=(t,\hat{x}^1_{[w_1]},...,\hat{x}^n_{[w_n]})$ are coordinates on the open subset $\hat{\pi}^{-1}(U)=\Spec(\R_{N,w}(C^\infty_M(U)))$. View $\hat{\phi}$ as a family of maps from the fibers of $t^{-1}(\tau)\subset \D_w(M,N)$ to $\{\tau\}\times \bR^n$: for $\tau\ne 0$, $\hat{\phi}|_{t^{-1}(\tau)}$ is a weighted re-scaling of the coordinate chart $\phi$, while $\hat{\phi}|_{t^{-1}(0)}$ is the coordinate chart for $\nu_w(M,N)$ from the proof of Theorem \ref{t:wt-norm}. From this description it is clear that $\hat{\phi}$ is injective. The image of $\hat{\phi}$ is the open subset
\[ \{(c,c_1,...,c_n)\in \bR^{n+1}\mid (c^{w_1}c_1,...,c^{w_n}c_n)\in \phi(U)\} \subset \bR^{n+1}.\]
By the definition of $\R_{N,w}(C^\infty_M(U))$ and Hadamard's lemma, the elements of $\R_{N,w}(C^\infty_M(U))$ are smooth functions of the variables $\{t\}\cup \{\hat{x}^i_{[w_i]}\mid 1\le i \le n\}$ and, similar to Theorem \ref{t:wt-norm}, $C^\infty_{\D_w(M,N)}(\hat{\pi}^{-1}(U))$ is smoothly generated by $\{t\}\cup \{\hat{x}^i_{[w_i]}\mid 1\le i\le n\}$, which are therefore local coordinates for $\hat{\pi}^{-1}(U)$. By Corollary \ref{c:spec-manifold}, $\D_w(M,N)$ is a manifold.
\end{proof}
\begin{proposition}
\label{p:Rees-alg-dyn}
If $F \in C^\infty_{\D_w(M,N)}$ is homogeneous of degree $k\in \bZ$ with respect to the zoom action, then $F=\hat{f}_{[k]}$ where $f=F|_{t=1}$. The Rees algebra $\R_{N,w}(C^\infty_M)$ is the sub-sheaf of $C^\infty_{\D_w(M,N)}$ consisting of finite sums of smooth functions that are homogeneous with respect to the zoom action.
\end{proposition}
\begin{proof}
Let $F\in C^\infty_{\D_w(M,N)}$ be homogeneous of degree $k$ with respect to the zoom action. On the open dense subset $\bR^\times \times M\subset \D_w(M,N)$, the zoom action is simply $\hat{\kappa}_\lambda(\tau,m)=(\lambda^{-1}\tau,m)$, hence on the intersection of the domain of $F$ with this open dense subset, $F$ must agree with $t^{-k}f$ where $f=F|_{t=1}\in C^\infty_M$. Restricting $f$ to the domain of a weighted coordinate chart $U$ and using Lemma \ref{l:Ik-dyn}, it follows that $f\in \I_k$. Thus $\hat{f}_{[k]}$ is defined, and since $F$, $\hat{f}_{[k]}$ agree on an open dense subset, they must be equal.
\end{proof}
The map $t\colon \D_w(M,N)\rightarrow \bR$ is a submersion intertwining $\hat{\kappa}$ with the weight $-1$ scaling action of $\bR^\times$ on $\bR$. The weighted deformation space construction extends to a functor $\D_w$ from the category of weighted manifolds to the category of $\bR^\times$-manifolds equipped with a $\bR^\times$-equivariant submersion to $\bR$, and moreover we have the following.
\begin{theorem}
\label{t:full-faithful}
The functor $\D_w$ is fully faithful.
\end{theorem}
\begin{proof}
Let $\M=(M,N,\I_\bullet)$, $\M'=(M',N',\I'_\bullet)$ be weighted manifolds. To avoid confusion we will use primes $'$ to denote corresponding objects for $\M'$: $t'$, $\kappa'$, and so on. It is clear that $\D_w\colon \Hom(\M,\M')\rightarrow \Hom(\D_w(M,N),\D_w(M',N'))$ is injective, since if $h\in \Hom(\M,\M')$ is a weighted morphism then $\D_w(h)|_{t^{-1}(1)}=h$. 

Conversely let $H\colon \D_w(M,N)\rightarrow \D_w(M',N')$ be a morphism in the category of $\bR^\times$-manifolds equipped with a $\bR^\times$-equivariant submersion to $\bR$. By assumption, for each $\tau \in \bR$, $H$ restricts to a smooth map $H_\tau$ between the fibers $t^{-1}(\tau)\subset \D_w(M,N)$ and $(t')^{-1}(\tau)\subset \D_w(M',N')$. By $\bR^\times$-equivariance and continuity, $H$ satisfies
\[ H(\lim_{\lambda\rightarrow a} \hat{\kappa}_\lambda(p))=\lim_{\lambda\rightarrow a} \hat{\kappa}'_\lambda H(p) \]
for $a=0,\infty$ when the limit on the left-hand-side exists. For $p \in t^{-1}(0)=\nu_w(M,N)$ and $a=0$ we deduce that $H_0$ maps $N$ to $N'$. For $p\in \{\tau\}\times N\subset t^{-1}(\tau)$ and $a=\infty$ we deduce that $H_\tau(p)=H_0(p)$. Thus $h:=H_1$ is a map of pairs $(M,N)\rightarrow (M',N')$. It also follows from the fact that $H$ is a $\bR^\times$-equivariant map of manifolds over $\bR$, that on the open dense subset $\bR^\times \times M\subset \D_w(M,N)$, $H$ is given by $H(\tau,m)=(\tau,h(m))$.

Let $f'\in \I_k'$. By $\bR^\times$-equivariance, $H^*(\hat{f}'_{[k]})$ is homogeneous of degree $k$, hence by Proposition \ref{p:Rees-alg-dyn}, $H^*(\hat{f}'_{[k]})=\hat{f}_{[k]}$ for some $f \in \I_k$. Thus $h^*\I_k'\subset \I_k$ and $h$ is a weighted morphism. As $\D_w(h)$, $H$ agree on the open dense subset $\bR^\times \times M\subset \D_w(M,N)$, continuity implies $\D_w(h)=H$.
\end{proof}

\begin{remark}
Let $(M,N,\I_\bullet)$ be a weighted manifold. Another useful construction is the \emph{weighted blow-up} of $M$ along $N$, which can be constructed as the quotient $(\D_w(M,N)-\bR\times N)/\bR^\times$. Note that if some of the weights are even, the $\bR^\times$ action is not quite free, and this space may have $\bZ/2\bZ$-orbifold singularities.
\end{remark}

\section{Fiber products}
In this section we discuss fiber products of weighted manifolds, largely following \cite{hudson2024multiplicative}. With the aim of keeping the article self-contained, we re-prove Hudson's result on fiber products for morphisms that are weighted transverse. We also prove that the weighted normal bundle and weighted deformation space constructions are compatible with fiber products.

\begin{definition}
Let $(M,N,\I_\bullet)$ be a weighted manifold. A \emph{weighted submanifold} is a submanifold $S\subset M$ such that every point $p \in S\cap N$ admits a weighted coordinate chart that is also a submanifold chart for $S$.
\end{definition}
\begin{remark}
Any submanifold of $M-N$ or of $N$ is weighted. Any submanifold that intersects $N$ transversely is weighted. If $M$ is trivially weighted along $N$, then a weighted submanifold is the same thing as a submanifold $S$ of $M$ that intersects $N$ cleanly (i.e. $S\cap N$ is smooth and $T(S\cap N)=TS\cap TN$).
\end{remark}
It is more or less immediate from the definition that a weighted submanifold $S$ inherits a weighting along $S\cap N$ such that $S\hookrightarrow M$ is a weighted morphism, and moreover $T_{S\cap N}S_{(i)}=TS\cap T_NM_{(i)}$. 
\begin{remark}
By \cite[Proposition 2.11]{hudson2024multiplicative}, if $r\colon M\rightarrow M$ is a weighted morphism satisfying $r\circ r=r$ then $S=r(M)$ is a weighted submanifold.
\end{remark}

\begin{definition}
A weighted morphism $f\colon (M,N,\I_\bullet)\rightarrow (M',N',\I'_\bullet)$ is a \emph{weighted submersion} if $f$ is a submersion and the induced morphism $\gr(\d_Nf)\colon \gr(T_NM)\rightarrow \gr(T_{N'}M')$ between the associated graded vector bundles is surjective.
\end{definition}
\begin{remark}
\label{r:wt-sub-loc-sec}
By \cite[Theorem 2.22]{hudson2024multiplicative}, a weighted morphism $f$ is a weighted submersion if and only if for every $p \in M$ there exist weighted coordinate charts $U\ni p$ and $U'\supset f(U)$ respectively, for which $f|_U$ is a coordinate projection. It follows that weighted submersions admit weighted local sections passing through any point of the domain, and conversely a weighted morphism admitting weighted local sections through any point of the domain is a weighted submersion.
\end{remark}

\begin{definition}
Let $(M,N,\I_\bullet)$, $(M',N',\I'_\bullet)$, $(M'',N'',\I''_\bullet)$ be weighted manifolds. Weighted morphisms $f\colon M\rightarrow M''$, $f'\colon M'\rightarrow M''$ are \emph{weighted transverse} if they are transverse and the induced vector bundle morphisms $\gr(\d_Nf),\gr(\d_{N'}f')$ are transverse.
\end{definition}
In detail, the transversality condition for the associated graded bundles says that for each $k=0,1,2,...,r$, and for all $(p,p')\in N\times_{N''}N'$ with $p''=f(p)=f'(p')$,
\[ \gr^k(\d_pf)(\gr^k(T_pM))+\gr^k(\d_{p'}f')(\gr^k(T_{p'}M'))=\gr^k(T_{p''}M'').\]
In particular for $k=0$, $f|_N,f'|_{N'}$ are transverse as maps to $N''$, hence $N\times_{N''}N'$ is a smooth submanifold of $M\times_{M''}M'$.

\begin{theorem}[cf. Theorem 2.27 of \cite{hudson2024multiplicative}]
\label{t:fiber-product}
Let $(M,N,\I_\bullet)$, $(M',N',\I'_\bullet)$, and $(M'',N'',\I''_\bullet)$ be weighted manifolds. Let $f\colon M\rightarrow M''$, $f'\colon M'\rightarrow M''$ be weighted morphisms that are weighted transverse. The fiber product $M\times_{M''}M'$ is a weighted submanifold of the product $M\times M'$.
\end{theorem}
\begin{proof}
$M\times M'$ carries the product weighting $\I^{M\times M'}_{k}=C^\infty_{M\times M'}\cdot \sum_{i+j=k}\pi_1^*\I^M_{i}\cdot \pi_2^*\I^{M'}_{j}$. Let $(p,p')\in N\times_{N''}N'=M\times_{M''}M'\cap (N\times N')$, and we must exhibit weighted coordinates around $(p,p')$ that are also submanifold coordinates for $M\times_{M''}M'\subset M\times M'$. Choose a weighted coordinate system $x^1,...,x^{n+n'}$ for $M\times M'$ near $(p,p')$. For each $k$, let $J_k\subset \{1,...,n+n'\}$ be the subset such that $j\in J_k \Leftrightarrow w_j=k$, in other words $\{x^j\mid j \in J_k\}$ is the subset of the coordinates that have weight $k$. Let $z^1,...,z^{n''}$ be weighted coordinates for $M''$ near $p''=f(p)=f'(p')$ with weights $w_1'',...,w_{n''}''$. Let
\[ y^j=f^*z^j-(f')^*z^j, \qquad j=1,...,n''.\]
Since $f,f'$ are weighted, $y^j\in \I^{M\times M'}_{w_j''}$. The functions $\{y^j\mid j=1,...,n''\}$ cut out the fiber product $M\times_{M''}M'$ near $(p,p')\in M\times M'$. Weighted transversality implies that the map induced by $\d_pf\oplus (-\d_{p'}f')\colon T_pM\oplus T_{p'}M'\rightarrow T_{p''}M''$ on the $k$-th pieces of the associated graded spaces is surjective, thus the dual map induced by $((\d_pf)^*,-(\d_{p'}f')^*)\colon T_{p''}^*M''\rightarrow T_p^*M\oplus T_{p'}^*M'$ on the $k$-th pieces of the associated graded is injective, and hence the set of images of the differentials
\[ \{\d_{(p,p')}y^j=((\d_pf)^*\d z^j,-(\d_{p'}f')^*\d z^j)\mid w_j''=k, j=1,...,n''\} \]
in $\gr^k(T^*_pM\oplus T^*_{p'}M')$ is linearly independent. By the replacement lemma in linear algebra, for each $k$ there is a subset $\tilde{J}_k\subset J_k$, such that the images of the differentials
\begin{equation} 
\label{e:basis-wt-k}
\{\d_{(p,p')} y^j\mid w_j''=k,j=1,...,n''\}\cup \{\d_{(p,p')} x^j\mid j \in \tilde{J}_k\} 
\end{equation}
in $\gr^k(T^*_pM\oplus T^*_{p'} M')$ form a basis of $\gr^k(T^*_pM\oplus T^*_{p'} M')$. Let $\tilde{J}=\cup_k \tilde{J}_k$. Then $\{y^j\mid j=1,...,n''\}\cup \{x^j\mid j \in \tilde{J}\}$ is a system of coordinates having the correct weights (with multiplicities) to form a weighted coordinate chart, hence by Proposition \ref{p:criterion-for-wt-chart} these coordinates form a weighted coordinate chart near $(p,p')$; by construction this is also a system of submanifold coordinates for $M\times_{M''}M'\subset M\times M'$.
\end{proof}

\begin{proposition}
\label{p:fiber-prod-def-space}
Let $(M,N,\I_\bullet)$, $(M',N',\I'_\bullet)$, and $(M'',N'',\I''_\bullet)$ be weighted manifolds. Let $f\colon M\rightarrow M''$, $f'\colon M'\rightarrow M''$ be weighted morphisms that are weighted transverse. Then $\nu_w(f),\nu_w(f')$ (resp. $\D_w(f),\D_w(f')$) are transverse and the canonical morphisms are isomorphisms
\begin{align*} 
\nu_w(M\times_{M''}M',N\times_{N''}N')&\xrightarrow{\sim} \nu_w(M,N)\times_{\nu_w(M'',N'')}\nu_w(M',N'),\\
\D_w(M\times_{M''}M',N\times_{N''}N')&\xrightarrow{\sim} \D_w(M,N)\times_{\D_w(M'',N'')}\D_w(M',N'),
\end{align*} 
in the category of graded bundles, resp. the category of $\bR^\times$-manifolds equipped with a $\bR^\times$-equivariant submersion to $\bR$.
\end{proposition}
\begin{proof}
The linearization of $\nu_w(f)\colon \nu_w(M,N)\rightarrow \nu_w(M'',N'')$ along $N$ is obtained by applying the normal bundle functor to $\nu_w(f)$, and yields the morphism induced on associated graded bundles:
\[ \nu(\nu_w(f))=\gr(\nu(f))\colon \nu(\nu_w(M,N),N)\simeq \gr(\nu(M,N))\rightarrow \nu(\nu_w(M'',N''),N'')\simeq \gr(\nu(M'',N'')),\]
and similarly for the linearization of $\nu_w(f')$. By hypothesis, the linearizations $\gr(\nu(f))$, $\gr(\nu(f'))$ are transverse, and hence $\nu_w(f),\nu_w(f')$ are transverse. Recall that $\D_w(M'',N'')$ has a surjective submersion to $\bR$, with fibers $\D_w(M'',N'')=\nu_w(M'',N'')\sqcup (\bR^\times \times M'')$ (and similarly for $M,M'$). Transversality of $f,f'$ implies that $\D_w(f),\D_w(f')$ are transverse over the subset $\bR^\times \times M''\subset \D_w(M'',N'')$, and as just explained, the morphisms induced on $0$-fibers $\nu_w(f),\nu_w(f')$ are also transverse; since $\D_w(M,N),\D_w(M',N'),\D_w(M'',N'')$ are equipped with compatible surjective submersions to $\bR$, it follows that $\D_w(f),\D_w(f')$ are also transverse. 

Applying the functor $\nu_w$ (resp. $\D_w$) to the weighted morphisms given by projection to the first and second factors $M\times_{M''}M'\rightarrow M,M'$ yields the canonical morphisms in the statement of the proposition.

The linearization of $\nu_w(M\times_{M''}M',N\times_{N''}N')$ along $N\times_{N''}N'$ is the associated graded 
\[ \gr(\nu(M\times_{M''}M',N\times_{N''}N'))\simeq \gr(\nu(M,N))\times_{\gr(\nu(M'',N''))}\gr(\nu(M',N')) \]
and it follows that the morphism of graded bundles $\nu_w(M\times_{M''}M',N\times_{N''}N')\rightarrow \nu_w(M,N)\times_{\nu_w(M'',N'')}\nu_w(M',N')$ induces an isomorphism on associated graded bundles, and hence is itself an isomorphism. The morphism $\D_w(M\times_{M''}M',N\times_{N''}N')\rightarrow\D_w(M,N)\times_{\D_w(M'',N'')}\D_w(M',N')$ is trivially a diffeomorphism over the open subset $\bR^\times \times M''\subset \D_w(M'',N'')$, and also over the $0$-fiber $\nu_w(M'',N'')$ as just explained; these observations together with the fact that $\D_w(M\times_{M''}M',N\times_{N''}N')$, $\D_w(M,N)\times_{\D_w(M'',N'')}\D_w(M',N')$ are equipped with compatible surjective submersions to $\bR$, together imply that the map $\D_w(M\times_{M''}M',N\times_{N''}N')\rightarrow\D_w(M,N)\times_{\D_w(M'',N'')}\D_w(M',N')$ is a diffeomorphism.
\end{proof}

\section{Multiplicative weightings}
In this section we recall the definition of multiplicative weightings of Lie groupoids along their units from \cite{hudson2024multiplicative}, and prove that in this context the weighted normal bundle and weighted deformation space are themselves Lie groupoids.

Let $G\rightrightarrows M$ be a Lie groupoid equipped with a weighting along $M$. Equip $M$ with the trivial weighting for which the submanifold $N=M$. The source map $s\colon G\rightarrow M$ and the target map $t\colon G \rightarrow M$ are weighted submersions, and the unit embedding $u\colon M\hookrightarrow G$ is also a weighted morphism. By Theorem \ref{t:fiber-product}, the space of composable arrows $G^{(2)}=G\times_M G$ is a weighted submanifold of the product $G\times G$.

\begin{definition}
\label{d:mult-weighting}
Let $G\rightrightarrows M$ be a Lie groupoid. A \emph{multiplicative weighting} of $G$ along $M$ is a weighting of $G$ along $M$ such that the groupoid multiplication
\[ m\colon G^{(2)}\rightarrow G \]
is a weighted morphism.
\end{definition}
\begin{proposition}
\label{p:inversion-weighted}
Let $G\rightrightarrows M$ be a Lie groupoid equipped with a multiplicative weighting along $M$. Then the inversion map is a weighted morphism from $G$ to itself.
\end{proposition}
\begin{proof}
Let $\tau\colon G\times_{s,M,t}G\rightarrow G\times_{t,M,t}G$ be the map $\tau(g_1,g_2)=(g_1,g_1g_2)$ with components $(\pr_1,m)$, where $\pr_1$ is projection to the first factor. Since $\pr_1,m$ are weighted morphisms, $\tau$ is a weighted morphism. Since $G$ is a Lie groupoid, $\tau$ is a diffeomorphism. The fiber products $G\times_{s,M,t}G$ and $G\times_{t,M,t}G$ are weighted along the obvious diagonally embedded copies of $M$ and have the same weight sequence (given by doubling the multiplicity of each non-zero weight for the weighting of $G$ along $M$). Thus the injective induced maps on filtered pieces $T_M(G\times_{s,M,t}G)_{(i)}\rightarrow T_M(G\times_{t,M,t}G)_{(i)}$ are isomorphisms, hence by Remark \ref{r:wt-sub-loc-sec}, $\tau$ admits weighted local sections, and consequently $\tau^{-1}$ is a weighted morphism. The result then follows from the observation that the inversion map factors as the composition of weighted morphisms
\[ G\xrightarrow{\iota_1} G\times_{t,M,t}G\xrightarrow{\tau^{-1}}G\times_{s,M,t}G\xrightarrow{\pr_2}G \]
where $\iota_1(g)=(g,u(t(g)))$.
\end{proof}
\begin{remark}
The definition of multiplicative weighting was introduced in \cite{hudson2024multiplicative}, where the more general case of weightings along Lie subgroupoids is studied. There is another approach (discussed in \cite{LMweightings}) to multiplicative weightings of Lie groupoids in terms of the jet subbundle $Q\subset T_rG$ from Remark \ref{r:jets}, namely as weightings such that $Q$ is a Lie subgroupoid of $T_rG$. The two approaches are shown to be equivalent in \cite{hudson2024multiplicative}.
\end{remark}

\begin{theorem}
\label{t:def-sp-grpd}
Let $G\rightrightarrows M$ be a Lie groupoid equipped with a multiplicative weighting along $M$. Then, applying the weighted normal bundle functor $\nu_w$ (resp. weighted deformation space functor $\D_w$) yields Lie groupoids
\[ \nu_w(G,M)\rightrightarrows M, \qquad \D_w(G,M)\rightrightarrows \bR\times M.\]
The Lie groupoid $\nu_w(G,M)$ is a bundle of simply connected nilpotent Lie groups over $M$, with groupoid multiplication intertwining the $(\bR,\cdot)$-monoid action. The Lie groupoid $\D_w(G,M)$ is a smooth family of Lie groupoids $t\colon \D_w(G,M)\rightarrow \bR$, interpolating between $t^{-1}(1)=G$ and $t^{-1}(0)=\nu_w(G,M)$, with the groupoid multiplication intertwining the $\bR^\times$ zoom action.
\end{theorem}
\begin{proof}
Most of the claims follow from functoriality of the weighted normal bundle/weighted deformation space constructions, Proposition \ref{p:fiber-prod-def-space} and Proposition \ref{p:inversion-weighted}. For example
\[ \D_w(G,M)^{(2)}:=\D_w(G,M)\times_{\bR\times M}\D_w(G,M)\simeq \D_w(G^{(2)},M)\xrightarrow{\D_w(m)}\D_w(G,M) \]
is the Lie groupoid multiplication of $\D_w(G,M)$, and similarly for $\nu_w(G,M)$. The source and target maps of $\nu_w(G,M)$ are obtained by applying the functor $\nu_w$ to the source and target maps of $G$, and the resulting maps $\nu_w(G,M)\rightarrow \nu_w(M,M)\simeq M$ both agree with the bundle projection to $M$, hence $\nu_w(G,M)$ is a bundle of Lie groups, which are simply connected since the fibers of $\nu_w(G,M)$ are diffeomorphic to $\bR$-vector spaces. The $(\bR,\cdot)$-monoid action on the fibers intertwining the group multiplication differentiates to a $\bZ_{>0}$-grading on the Lie algebras of the fibers, which are therefore nilpotent (any finite dimensional $\bZ_{>0}$-graded Lie algebra is nilpotent for degree reasons).
\end{proof}

When $G$ is a Lie groupoid equipped with the standard weighting along $M$, $\nu_w(G,M)=\nu(G,M)=A$ is the Lie algebroid of $G$, viewed as a Lie groupoid over $M$ with source and target both given by the bundle projection $A\rightarrow M$, and the groupoid multiplication given by addition in the fibers of $A$; in other words, $A\rightarrow M$ is a bundle of abelian Lie groups. The Lie groupoid $\D_w(G,M)=\D(G,M)$ is the \emph{adiabatic Lie groupoid} associated to $G$. A special case is $G=\tn{Pair}(M)=M\times M$ the pair groupoid of $M$, for which $\nu_w(G,M)=A=TM$ and $\D_w(G,M)$ is the \emph{tangent groupoid} of $M$. 

We will see below that for a non-standard multiplicative weighting of $G$ along $M$, the fibers of $\nu_w(G,M)$ are non-abelian nilpotent Lie groups in general; this occurs already for non-standard weightings of the pair groupoid $\tn{Pair}(M)=M\times M$ along the diagonal $M$. For appropriate choices of non-standard weightings, the \emph{weighted pair groupoid} $\D_w(\tn{Pair}(M),M)$ (and more generally, the \emph{weighted adiabatic groupoid} $\D_w(G,M)$) recovers the Lie groupoids discussed for instance in \cite{van2019groupoid}, key to the construction of non-standard pseudo-differential calculi used in the study of hypoelliptic operators.

\section{Lie filtrations and dg weightings}
In this section we recall the notion of Lie filtrations $A_\bullet$ of a Lie algebroid $A\rightarrow M$. We define dg weightings of NQ-manifolds and prove that the data of a Lie filtration of $A_\bullet$ is equivalent to the data of a dg weighting of the NQ-manifold $A[1]$ (the NQ-manifold whose corresponding sheaf of dg algebras is $(\Omega_A,\d_A)$, the Chevalley-Eilenberg-de Rham complex of $A$).

\begin{definition}
Let $A$ be a Lie algebroid over a manifold $M$. A \emph{filtration} $A_\bullet$ of $A$ is a sequence of nested subbundles
\[ A_1\subset A_2\subset \cdots \subset A_r=A. \]
(By convention $A_0=0$ and $A_j=A$ for $j>r$.) A \emph{Lie filtration} of $A$ is a filtration $A_\bullet$ with the property that $[\A_i,\A_j]\subset \A_{i+j}$, where $\A_i$ is the sheaf of local sections of $A_i$.
\end{definition}
There are many interesting examples already for the case $A=TM$. For instance if $M$ is a contact manifold with hyperplane distribution $H$, setting $A_1=H$, $A_2=A=TM$ is an example of a Lie filtration. More generally, in sub-Riemannian geometry one studies `bracket-generating distributions' $A_1\subset A=TM$ with the property that, for each $k$, iterated Lie brackets of sections of $A_1$ of depth $\le k$ span a subbundle $A_k$, and $A_r=TM$ for some $r<\infty$; in this case $A_\bullet$ is a Lie filtration. The following result is well-known.
\begin{proposition}
\label{p:nilp-Lie}
Let $A$ be a Lie algebroid equipped with a Lie filtration $A_\bullet$. The Lie bracket of $A$ descends to a bracket on the fibers of the associated graded vector bundle $\gr(A)$, making the latter into a bundle of graded nilpotent Lie algebras.
\end{proposition}
\begin{proof}
In detail the bracket is given by
\[ [X+\A_{i-1},Y+\A_{j-1}]=[X,Y]+\A_{i+j-1} \]
for $X\in \A_i,Y\in \A_j$. One verifies that the bracket is well-defined and $C^\infty_M$-bilinear, hence induces a Lie bracket on the fibers of $\gr(A)$. The resulting bundle of Lie algebras is $\bZ_{>0}$-graded and is thus nilpotent for degree reasons.
\end{proof}
For example in the case of a contact manifold mentioned above, the nilpotent Lie algebra in Proposition \ref{p:nilp-Lie} is the Heisenberg Lie algebra. In more elaborate examples, the isomorphism class of the nilpotent Lie algebra can vary, see for instance \cite[Example 2.5]{beschastnyi2026sub}.

In the remainder of this section we generalize weightings to NQ-manifolds and reformulate Lie filtrations of $A$ dually in terms of weightings of the NQ-manifold $A[1]$ associated to the Chevalley-Eilenberg-de Rham complex $(\Omega_A,\d_A)$ of $A$. The dual characterization is essentially a reformulation, in terms of weightings of NQ-manifolds, of a result contained in \cite[Theorem 5.15(c)]{hudson2024multiplicative} (formulated there in terms of `infinitesimally multiplicative' weightings of Lie algebroids). This discussion is not needed in the following two sections.

Let 
\[ V=\bigoplus_{j=0}^p V_j=V_0\oplus V_{>0} \] 
be a finite-dimensional $\bZ_{\ge 0}$-graded real vector space with a fixed homogeneous basis, and let $n_0=\dim(V_0)$, $n=\dim(V)$. For an open $U\subset V_0$ define
\[ C^\infty_V(U)=C^\infty(U)\otimes \Sym(V_{>0}^*),\]
where $\Sym(V_{>0}^*)$ denotes the \emph{graded symmetric algebra}, i.e. the ordinary symmetric algebra of the even degree subspace of $V_{>0}^*$ tensored with the exterior algebra of the odd degree subspace of $V_{>0}^*$; $C^\infty_V$ is a sheaf of graded-commutative rings on $V_0$. Let $x^1,...,x^n$ be the homogeneous coordinates on $V$ corresponding to the basis. Fix a \emph{weight sequence} $0\le w_1\le \cdots \le w_n=r$. The weight sequence determines a nested sequence of homogeneous ideals
\begin{equation} 
\label{e:wt-loc-form-graded}
I_k=(x^s=(x^1)^{s_1}\cdots (x^n)^{s_n}\mid w\cdot s\ge k)
\end{equation}
of $C^\infty_V(U)$.

Let $(\M,M,C^\infty_\M)$ be a $\bZ_{\ge 0}$-\emph{graded manifold}: an ordinary manifold $M$ together with a sheaf of $\bZ_{\ge 0}$-graded commutative rings $C^\infty_\M$ on $M$, such that around each point in $M$ there is an open neighborhood $U\subset M$ and an isomorphism $\phi$ from $C^\infty_\M(U)$ to $C^\infty_V(U)$ for a $\bZ_{\ge 0}$-graded vector space with basis $V$ as above. By analogy with ordinary manifolds, we refer to $\phi$ as a coordinate chart of $\M$, and we refer to the corresponding local sections $x^1,...,x^n\in C^\infty_\M(U)$ as local coordinates on $\M$.
\begin{example}
\label{e:grad-mfld}
Let $A\rightarrow M$ be a vector bundle. Let $A[1]$ be the $\bZ_{\ge 0}$-graded manifold with underlying ordinary manifold $M$ and sheaf of $\bZ_{\ge 0}$-graded commutative rings $C^\infty_{A[1]}=\Omega_A$, the sheaf of sections of the exterior algebra $\wedge A^*=\Sym(A[1]^*)$. Local coordinates for $A[1]$ are obtained by choosing local coordinates on an open $U\subset M$ (these become the coordinates of degree $0$) together with a local frame of $A^*|_U$ (these become the coordinates of degree $1$).
\end{example}
\begin{definition}
Let $(\M,M,C^\infty_\M)$ be a $\bZ_{\ge 0}$-graded manifold. A \emph{graded weighting} of $\M$ is a nested sequence of ideal sheaves
\[ C^\infty_\M=\I_0\supset \I_1\supset \cdots  \]
such that for every $p \in M$, there exists a coordinate chart $(U,\phi=(x^1,...,x^n))$ around $p$ and a weight sequence $w=(w_1,...,w_n)$, such that in $\phi$-coordinates, $\{\I_k(U)\}$ is of the form \eqref{e:wt-loc-form-graded}.
\end{definition}
The ideal $\I_1$ is the vanishing ideal of a $\bZ_{\ge 0}$-graded submanifold $\N$, and one speaks of a weighting of $\M$ along $\N$. We will focus on the case where $\N=M$ is the underlying manifold of the $\bZ_{\ge 0}$-graded manifold $\M$.
\begin{proposition}
Let $\M=A[1]$ for a vector bundle $A\rightarrow M$. The data of a weighting of $\M$ along $M$ is equivalent to a sequence of subbundles
\[ A_1\subset \cdots \subset A_r=A,\]
where $A_k$ is the annihilator of the elements of $\Omega^1_A$ of weight $>k$.
\end{proposition}
\begin{proof}
Let $\M=A[1]$ be weighted. It follows from the local normal form of a weighting, equation \eqref{e:wt-loc-form-graded}, that $A_k$ is a smooth vector subbundle of $A$. Thus a weighting determines a sequence of nested subbundles $0_M=A_0\subset A_1\subset \cdots \subset A_r=A$ (compare also equation \eqref{e:normal-bun-filtration}). Conversely, given a sequence of subbundles as in the statement, let $\J_j=\I_{A_{j-1}}\subset \Omega_A$ be the ideal generated by the annihilator of $A_{j-1}$ in $\Omega_A^1$, and set $\I_k=\sum_i\sum_{j_1+\cdots+j_i=k}\J_{j_1}\cdots\J_{j_i}$ (compare Remark \ref{r:nested-submanifolds}). Choosing a local frame of $A$ (as in Example \ref{e:grad-mfld}) compatible with the filtration $A_\bullet$ yields local weighted coordinates, and thus a sequence of subbundles $A_1\subset \cdots \subset A_r=A$ determines a weighting. The constructions are inverse to each other.
\end{proof} 

Let $(\M,M,C^\infty_\M,\D)$ be an \emph{NQ-manifold}: a $\bZ_{\ge 0}$-graded manifold $(\M,M,C^\infty_\M)$ together with a \emph{cohomological vector field} $\D$, i.e. a derivation $\D$ of $C^\infty_\M$ of degree $1$ such that $\D^2=0$.
\begin{example}
Let $A\rightarrow M$ be a Lie algebroid. Let $\M=A[1]$, $C^\infty_\M=\Omega_A$. Let $\D=\d_A$ be the differential of the Chevalley-Eilenberg-de Rham complex. Then $(A[1],M,\Omega_A,\d_A)$ is an NQ-manifold.
\end{example}
\begin{definition}
Let $(\M,M,C^\infty_\M,\D)$ be an NQ-manifold. A \emph{dg weighting} of $(\M,M,C^\infty_\M,\D)$ is a graded weighting $\I_\bullet$ of $(\M,M,C^\infty_\M)$ such that each $\I_k$ is a differential ideal, that is, $\D\I_k\subset \I_k$.
\end{definition}
\begin{theorem}
Let $A\rightarrow M$ be a Lie algebroid. Let $\I_\bullet$ be a graded weighting of $(A[1],M,\Omega_A)$. Let $A_k\subset A$ be the annihilator of the elements of $\Omega^1_A$ of weight $>k$. Then $\I_\bullet$ is a dg weighting of $(A[1],M,\Omega_A,\d_A)$ if and only if $A_\bullet$ is a Lie filtration.
\end{theorem}
\begin{proof}
Let $\varrho\colon A\rightarrow TM$ denote the anchor map. For $k=1,2,...,r$ let $\{X_k^1,...,X_k^{n_k}\}$ be sections (over an open $U\subset M$) of $A_k$ that descend to a local frame of $A_k/A_{k-1}$, hence $\{X_k^i\mid 1\le k\le r,1\le i\le n_k\}$ is a local frame of $A$ adapted to the filtration $A_\bullet$. Let $\{\alpha^k_i\}$ be the dual local frame of $A^*$. Then $\alpha^k_i$ annihilates $A_{k-1}$ hence $\alpha^k_i\in \I_k$. There are functions $f_{ll'}^{ii'}=-f_{l'l}^{i'i}\in C^\infty_M$ such that
\begin{equation} 
\label{e:dAalphaki}
\d_A\alpha^k_i=-\frac{1}{2}\sum_{l,l'}\sum_{j,j'} f_{ll'}^{jj'}\alpha^l_j\wedge \alpha^{l'}_{j'}.
\end{equation}
Evaluate both sides on a pair $(X_m^j,X_{m'}^{j'})$ for which $m+m'<k$. On the left hand side use the usual formula $\d_A\alpha(X,X')=\varrho(X)\alpha(X')-\varrho(X')\alpha(X)-\alpha([X,X'])$ and the fact that $\alpha^k_i$ annihilates $A_{k-1}$, while on the right hand side use duality $\alpha^l_j(X_{l'}^{j'})=\delta^l_{l'}\delta_j^{j'}$, to obtain
\begin{equation} 
\label{e:dAalphaki2}
\alpha^k_i([X_m^j,X_{m'}^{j'}])=f_{mm'}^{jj'}. 
\end{equation}
Suppose $\I_\bullet$ is a dg weighting. Then $\d_A\alpha^k_i \in \I_k$ and thus the terms in \eqref{e:dAalphaki} with $l+l'<k$ vanish, hence $f_{ll'}^{jj'}=0$ whenever $l+l'<k$. By equation \eqref{e:dAalphaki2}, $\alpha^k_i([X_m^j,X_{m'}^{j'}])=0$. As this holds for all $\alpha^k_i$ with $k>m+m'$, we deduce that $[X_m^j,X_{m'}^{j'}]$ is a section of $A_{m+m'}$, hence $A_\bullet$ is a Lie filtration. Conversely if $A_\bullet$ is a Lie filtration then $[X_m^j,X_{m'}^{j'}]$ is a section of $A_{m+m'}$ and the left hand side of \eqref{e:dAalphaki2} vanishes when $m+m'<k$, hence the terms of \eqref{e:dAalphaki} with $l+l'<k$ vanish, and it follows that $\I_\bullet$ is a dg weighting.
\end{proof}

\section{The multiplicative weighting associated to a Lie filtration}\label{s:mult-weighting-from-Lie}
Throughout this section we fix a Lie groupoid $(t,s)\colon G\rightrightarrows M$ with Lie algebroid $A$. We show that a Lie filtration $A_\bullet$ of $A$ determines a multiplicative weighting $\I_\bullet$ of $G$ along $M$. By working with local Lie groupoids one has a similar result more generally when $A$ is not necessarily integrable.

We introduce some additional notation. Let $\A$ denote the sheaf of smooth sections of $A$, and if $A_1\subset \cdots \subset A_r=A$ is a filtration of $A$, let $\A_k$ denote the sheaf of smooth sections of $A_k$. If $X\in \A_k(U)$ let $X^L$ (resp. $X^R$) denote the corresponding left (resp. right) invariant vector field on $s^{-1}(U)\subset G$ (resp. $t^{-1}(U)\subset G$).

\begin{definition}
Let $A_\bullet$ be a filtration of $A$. For $k\le 0$ set $\I_k=C^\infty_G$. Inductively define $\I_1=\I_M\subset C^\infty_G$, and for $k>1$,
\[ \I_k=\{f\in \I_M\mid \forall i>0,\forall X\in \A_i, X^L f \in \I_{k-i}\}\subset C^\infty_G.\]
It is not difficult to verify by induction that $\I_\bullet$ is a nested sequence of ideals that we will refer to as the \emph{sequence of ideals $\I_\bullet\subset C^\infty_G$ associated to the filtration $A_\bullet$}.
\end{definition}
\begin{example}
For example $\I_2$ consists of those $f\in \I_M$ such that
\[ X \in \A_1 \Rightarrow X^L f|_M=0 \]
and $\I_3$ consists of those $f \in \I_M$ such that
\[ X_1,X_2\in \A_1,Y\in \A_2\Rightarrow X_1^LX_2^Lf|_M=0=Y^Lf|_M,\]
and so on.
\end{example}
\begin{remark}\label{r:left-right}
The same sequence of ideals $\I_\bullet$ results if right invariant vector fields are used instead of left invariant vector fields. For example the identity
\[ X^R=(X^R-X^L)+X^L \]
and the fact that $X^R-X^L$ is tangent to $M$ together imply that if $f, X^Lf \in \I_M$ then $X^Rf \in \I_M$ too. Similarly the identity
\[ X^RY^R=(X^R-X^L)(Y^R-Y^L)+(X^R-X^L)Y^L+(Y^R-Y^L)X^L+Y^LX^L \]
implies that if $f,X^Lf,Y^Lf,X^LY^Lf\in \I_M$ then $X^RY^Rf\in \I_M$ too. The general argument is similar.
\end{remark}

Let $\U\A$ denote the sheaf of associative algebras over $M$ such that for an open $U\subset M$, $\U\A(U)$ is the universal enveloping algebra of the Lie-Rinehart algebra $\A(U)$. We will refer to $\U\A$ as the universal enveloping algebra of $\A$ for short. For example if $A=TM$ then $\U\A$ is the sheaf of scalar differential operators. If $P\in \U\A(U)$ then the corresponding left (resp. right) invariant differential operator on $s^{-1}(U)$ (resp. $t^{-1}(U)$) is denoted $P^L$ (resp. $P^R$).

\begin{definition}
\label{d:A-order}
Let $A_\bullet$ be a filtration of $A$. If $X \in \A_k$ then we say that $X$ (or $X^L$ or $X^R$) has $A_\bullet$-\emph{order at most} $k$. For a monomial $P=X_1\cdots X_l$ in the universal enveloping algebra, we say that $P$ (or $P^L$ or $P^R$) has $A_\bullet$-\emph{order at most} $k$ if $X_1\in \A_{k_1},...,X_l\in \A_{k_l}$ and $k_1+\cdots+k_l\le k$. The definition is extended to arbitrary elements of $\U\A$ by declaring that the $A_\bullet$-order of a sum of monomials is at most the maximum of upper bounds for the $A_\bullet$-orders of the summands.
\end{definition}
\begin{remark}
Note that an element of $\U\A$ (even a monomial) has many different expressions as sums of monomials in sections of $\A$, and these may lead to different upper bounds for the $A_\bullet$-order. The `$A_\bullet$-order' itself is the minimum over all possible expressions. We will only ever need upper bounds for the `$A_\bullet$-order', and so Definition \ref{d:A-order} with its qualifier `at most' will suffice.
\end{remark}

\begin{proposition}
\label{p:operator-condition}
Let $\I_\bullet \subset C^\infty_G$ be the sequence of ideals associated to a filtration $A_\bullet$ of $A$. Then $f \in \I_k$ if and only if $P^Lf\in \I_M$ for every $P\in \U\A$ with $A_\bullet$-order at most $k-1$.
\end{proposition}
\begin{proof}
This follows inductively from the definitions of $\I_k$ and of $A_\bullet$-order.
\end{proof}
\begin{remark}\label{r:left-right2}
As in Remark \ref{r:left-right} there is an analogous statement involving the right invariant family of differential operators $P^R$ instead of $P^L$.
\end{remark}

\begin{theorem}
\label{t:wt-assoc-Lie-filt}
Let $A_\bullet$ be a Lie filtration. The associated sequence of ideals $\I_\bullet \subset C^\infty_G$ is a weighting of $G$ along $M$. Moreover, under the identification $A=\nu(G,M)$, the induced filtration $\nu(G,M)_{(\bullet)}$ coincides with the Lie filtration $A_\bullet$.
\end{theorem}
\begin{proof}
The theorem is a special case of \cite[Theorem 4.1]{LMsingularlie}, but for the sake of completeness we include the proof. To avoid distraction we focus on the case of a single source fiber so $M=\{e\}$, $\dim(G)=n$, $G$ is a Lie group, and $A=A_e$ is its Lie algebra; the general case is a parametrized version and is no more difficult. For $X \in A_e=T_eG$ we will use the same symbol $X$ for the left-invariant vector field $X^L$, omitting the superscript `$L$', and similarly for elements of the universal enveloping algebra. We seek coordinates $x^1,...,x^n$ that vanish at $e\in G$ and satisfy the following \emph{vanishing conditions}:
\begin{equation} 
\label{e:vanishing-conditions}
A_\bullet-\tn{order}(P)<w_i \quad \Rightarrow \quad Px^i|_e=0.
\end{equation}
Let $n_i=\tn{rank}(A_i)$. Choose a basis $X_1,...,X_n$ of $A$ such that $X_1,...,X_{n_1}\in A_1$, $X_1,...,X_{n_2}\in A_2$, and so on. If $s=(s_1,...,s_n)$ is a multi-index, then the operator $X^s=X_1^{s_1}\cdots X_n^{s_n}$ has $A_\bullet$-order at most $w\cdot s=w_1s_1+\cdots+w_ns_n$. Let us call such monomials, with $X_1,...,X_n$ appearing in lexicographic order, `standard monomials'. We will refer to the number of factors $s_1+\cdots+s_n$ appearing in a monomial $X^s$ as the \emph{length} of the monomial.

Any differential operator of $A_\bullet$-order at most $k$ can be written as a linear combination, with $C^\infty$-coefficients, of the standard monomials $X^s$ such that $w\cdot s\le k$. Indeed, any differential operator can certainly be written as a linear combination with $C^\infty$ coefficients of monomials in the $X_i$ that are not necessarily standard. The point is that if one starts with a monomial having $A_\bullet$-order $k$ that is \emph{not} in standard order, it may be replaced with the standard monomial built from the same factors, at the cost of error terms involving commutators. The Lie filtration condition tells us that these commutator error terms also have $A_\bullet$-order at most $k$. On the other hand, the error terms are monomials with strictly smaller length (since $XY-YX$ is a vector field), hence inductively they can be written as a linear combination of standard monomials of $A_\bullet$-order at most $k$. 

As a result it suffices to find coordinates satisfying the vanishing conditions \eqref{e:vanishing-conditions} for standard monomials:
\begin{equation}  
\label{e:vanishing-condition2}
w\cdot s<w_i \quad \Rightarrow \quad X^s x^i|_e=0.
\end{equation}
Choose an initial system of coordinates $\bar{x}^i$ satisfying
\[ X_i\bar{x}^j|_e=\delta_i^j.\]
These satisfy the vanishing condition \eqref{e:vanishing-condition2} for standard monomials of length $1$. In particular $\bar{x}^1,...,\bar{x}^{n_1} \in \I_1$ and $\bar{x}^{n_1+1},...,\bar{x}^{n_2}\in \I_2$ (since $X_i\bar{x}^j|_e=0$ for $1\le i\le n_1$ and $n_1+1\le j\le n_2$). Assume inductively (in $k$) that we have already found coordinates with the correct weights up to weight $k$ for some $k\ge 2$, that is, $\bar{x}^i \in \I_{w_i}$ for $i=1,...,n_k$ for some $k\ge 2$. We will modify the coordinates $\bar{x}^i$ with $n_k+1\le i\le n_{k+1}$ to ensure that they have the required weight $k+1$. For $i\le n_k$ let $x^i=\bar{x}^i$, unchanged. We seek new coordinates $x^i$, for $n_k+1\le i\le n_{k+1}$, that belong to $\I_{k+1}$. We attempt to find a change of coordinates of the form
\begin{equation} 
\label{e:xi-def}
x^i=\bar{x}^i+\sum_{|u|\ge 2,w\cdot u< w_i} c^i_u(x)^u, 
\end{equation}
where $u=(u_1,...,u_{n_k})$ is a multi-index, $|u|=u_1+\cdots+u_{n_k}$, and $(x)^u=(x^1)^{u_1}\cdots (x^{n_k})^{u_{n_k}}$; that is, we modify $\bar{x}^i$ by a sum of monomials in the variables $x^1,...,x^{n_k}$. The $c^i_u$ are constants to be determined. Since the added terms have quadratic and higher degree, this is a valid coordinate change sufficiently near $e\in G$. To ensure $x^i$ belongs to $\I_{k+1}$ we require $x^i$ satisfy the vanishing conditions
\[ w\cdot s<k+1 \quad \Rightarrow \quad X^sx^i|_e=0.\]
We claim these constraints uniquely determine the constants $c_{iu}$. To show this we argue inductively in the length $l$ of the standard monomial: suppose that the $c_{iu}$ with $|u|=u_1+\cdots+u_{n_k}<l$ have already been determined in such a way that the vanishing conditions involving standard monomials of length $<l$ are satisfied. For the inductive step, we must choose the constants $c^i_s$ with $|s|=l$ such that the vanishing conditions involving standard monomials of length equal to $l$ are also satisfied. This is easy: let $|s|=l$ (with $w\cdot s<k+1$), apply the standard monomial $X^s$ to both sides of the equation \eqref{e:xi-def} defining $x^i$, and evaluate at $e$:
\[ X^sx^i|_e=X^s\bar{x}^i|_e+\sum_{|s|>|u|\ge 2,w\cdot u< w_i} c^i_uX^s(x)^u|_e+c^i_sX^s(x)^s|_e \]
where observe that we dropped terms with $|u|>|s|$ since then $X^s(x)^u|_e=0$. Crucially we also dropped the terms with $|u|=|s|$ but $u\ne s$ as then $X^s(x)^u|_e=0$ too: when the length of the monomial $X^s$ matches the length of the monomial $(x)^u$, the only non-zero contributions (from the Leibniz rule) to $X^s(x)^u|_e$ come from expressions having precisely one vector field from the monomial $X^s$ applied to each factor in the monomial $(x)^u$, and if $s\ne u$ then at least one of these vanishes at $e$, thanks to the condition $X_i\bar{x}^j|_e=\delta_i^j$. Rearrange the resulting equation to solve for $c^i_s$:
\[ c^i_s=-a_s^{-1}\Big(X^s\bar{x}^i|_e+\sum_{|s|>|u|\ge 2,w\cdot u<w_i}c^i_uX^s(x)^u|_e\Big), \quad a_s=X^s(x)^s|_e.\] 
This completes the inductive step proving the claim, and also completes the proof that $\I_\bullet$ is a weighting. The claim regarding the filtration of the normal bundle $\nu(G,M)$ follows immediately from the definition of $\nu(G,M)_{(\bullet)}$ and $\I_\bullet$.
\end{proof}
\begin{definition}
Let $A_\bullet$ be a Lie filtration and $\I_\bullet\subset C^\infty_G$ the associated sequence of ideals. In light of Theorem \ref{t:wt-assoc-Lie-filt}, we refer to $\I_\bullet$ as the \emph{weighting (of $G$ along $M$) associated to the Lie filtration} $A_\bullet$.
\end{definition}
\begin{theorem}
\label{t:Lie-filt-wt-mult}
The weighting of $G$ along $M$ associated to a Lie filtration $A_\bullet$ is multiplicative.
\end{theorem}
\begin{proof}
We must show that $m\colon G^{(2)}=G\times_MG\rightarrow G$ is a weighted morphism. Recall how the weighting $\I^{(2)}_\bullet$ of $G^{(2)}$ is defined (Theorem \ref{t:fiber-product}): take the product weighting $\I^{G\times G}_\bullet$ on $G\times G$ and restrict to the weighted submanifold $G^{(2)}\subset G\times G$, resulting in a weighting of $G^{(2)}$ along $M\simeq M^{(2)}\subset G^{(2)}$. The ideal $\I^{G\times G}_k$ is generated by $\sum_{i+j=k} \pi_1^*\I_i\cdot \pi_2^*\I_j$; this is the weighting of $G\times G$ along $M\times M$ defined by the product Lie filtration $(A\times A)_k:=A_k\times A_k$, hence $\I^{G\times G}_1=\I_{M\times M}$ and for $k>1$ we have, inductively,
\[ \I^{G\times G}_k=\{f\in \I_{M\times M}\mid \forall i<k, \forall X\in \A_i, (X^R\times 0) f, (0\times X^L)f \in \I^{G\times G}_{k-i}\} \]
where $X^R\times 0$ (resp. $0\times X^L$) denotes the vector field on $G\times G$ given by $X^R$ on the first factor of the product $G\times G$ (resp. second factor of the product $G\times G$), and here we have chosen to use a mix of right and left invariant vector fields according to Remarks \ref{r:left-right}, \ref{r:left-right2}. This choice (of left/right) is convenient since $(X^R\times 0)$, $(0\times X^L)$ are tangent to $G^{(2)}$, hence the same conditions define the weighting $\I^{(2)}_\bullet$ of $G^{(2)}$, that is $\I^{(2)}_1=\I_{M^{(2)}}$ and for $k>1$ we have, inductively,
\[ \I^{(2)}_k=\{f\in \I_{M^{(2)}}\mid \forall i<k, \forall X\in \A_i, (X^R\times 0) f, (0\times X^L) f \in \I^{(2)}_{k-i}\}. \]
Now $m^*\I_1=m^*\I_M\subset \I_{M^{(2)}}=\I^{(2)}_1$ because $m\colon M^{(2)}\subset G^{(2)}\rightarrow M\subset G$ is the identity. Let $k>1$ and suppose inductively we have shown that $m^*\I_j\subset \I_j^{(2)}$ for $j<k$. Let $f\in \I_k$ and we must show $m^*f$ belongs to $\I^{(2)}_k$. If $X\in \A_i$ then
\[ (X^R\times 0)m^*f=m^*(X^Rf)\in m^*\I_{k-i}\subset \I^{(2)}_{k-i}\]
by induction, and similarly
\[ (0\times X^L)m^*f=m^*(X^Lf)\in m^*\I_{k-i}\subset \I^{(2)}_{k-i} \]
by induction. Thus $m^*f\in \I^{(2)}_k$ as required.
\end{proof}

\section{The Lie filtration associated to a multiplicative weighting}
We continue to use notation introduced in the previous section, and in particular $G\rightrightarrows M$ is a Lie groupoid with Lie algebroid $A$. In this section we prove that every multiplicative weighting of $G$ along $M$ comes from a unique Lie filtration, in the manner explained in the previous section. As a result there is a one-one correspondence between Lie filtrations of $A$, and multiplicative weightings of $G$ along $M$.

\begin{definition}
Let $\I_\bullet$ be a weighting of $G$ along $M$. For $k=1,...,r$, let $A_k=\nu(G,M)_{(k)}\subset \nu(G,M)=A$. We refer to $A_\bullet$ as the \emph{filtration of $A$ associated to the weighting $\I_\bullet$}. Let $\A_k$ be the sheaf of local sections of $A_k$.
\end{definition}

\begin{definition}
Let $\I_\bullet$ be a weighting of $G$ along $M$ and by convention set $\I_j=C^\infty_G$ for $j\le 0$. Let $P \in \U\A$. If $P^L \I_j\subset \I_{j-k}$ for all $j$, then $P^L$ is said to have $\I_\bullet$-\emph{order at most $k$}.
\end{definition}
\begin{theorem}
\label{t:mult-wt-Lie-filt}
Let $\I_\bullet$ be a multiplicative weighting of $G$ along $M$, and let $A_\bullet$ be the filtration of $A$ associated to the weighting $\I_\bullet$. Then
\begin{enumerate}[a)]
\item If $X \in \A$ and $1\le k\le r$, then $X\in \A_k$ if and only if $X^L$ has $\I_\bullet$-order at most $k$.
\item If $P \in \U\A$ is a monomial with $A_\bullet$-order at most $k$, then $P^L$ has $\I_\bullet$-order at most $k$.
\item $A_\bullet$ is a Lie filtration of $A$.
\item The weighting $\I_\bullet'$ defined by the Lie filtration $A_\bullet$ coincides with $\I_\bullet$.
\end{enumerate}
\end{theorem}
\begin{proof}
a) Suppose $X^L$ has $\I_\bullet$-order at most $k$. Recall that $A_k=\nu(G,M)_{(k)}$ consists of normal vectors such that the corresponding normal directional derivative annihilates $\I_{k+1}$. By definition of $\I_\bullet$-order, $f\in \I_{k+1}\Rightarrow X^Lf\in \I_M \Rightarrow X^Lf|_M=0$. Hence the normal projection $\nu(X^L)=X$ is a section of the subbundle $\nu(G,M)_{(k)}=A_k\subset A=\nu(G,M)$.

Conversely suppose $X \in \A_k$ and we must show that $X^L$ has $\I_\bullet$-order at most $k$. First, we claim that there exists a vector field $\hat{X}$ defined on an open subset of $G$ such that $\hat{X}\I_j\subset \I_{j-k}$ for all $j$, and such that the normal projection $\nu(\hat{X}|_M)=X$. To construct $\hat{X}$ one can work locally and then patch together using a partition of unity, so we may as well assume that the domain of the local section $X$ is an open subset $U\subset M$ contained in the intersection of a local weighted coordinate chart for $G$ with the unit space $M$. Elements of $\A_k(U)$ are $C^\infty(U)$-linear combinations of normal projections $\nu(\partial/\partial x^i|_M)$ of coordinate vector fields for coordinates $x^i$ of strictly positive weight at most $k$. Hence $X=\sum_{1\le w_i\le k} f_i\nu(\partial/\partial x^i|_M)$ for smooth functions $f_i$ on $U$. Extend $f_i$ to a smooth function $\hat{f}_i$ on the domain of the coordinate chart. Then $\hat{X}=\sum_{1\le w_i\le k} \hat{f}_i(\partial/\partial x^i)$ has the desired properties.

Let $f\in \I_j$ and we must show $X^Lf\in \I_{j-k}$. If $j\le k$ this is automatic since $\I_{j-k}=C^\infty_G$, so assume $j>k$. Recall $m\colon G^{(2)}\rightarrow G$ denotes the groupoid multiplication. We shall write $\exp(tX)$ for the 1-parameter family of local bisections of $G$ generated by $X$: it need not be defined globally for any $t>0$; however for any given $p \in M$, $p\exp(tX)$ is defined for $t$ sufficiently small (and is given by the time $t$ flow of $X^L$ applied to $p$), which suffices for our purposes. Observe that
\[ (X^Lf)(g)=\frac{\partial}{\partial t}\bigg|_0 f(gs(g)\exp(tX))=\frac{\partial}{\partial t}\bigg|_0 (m^*f)(g,s(g)\exp(tX)).\]
By multiplicativity 
\[ m^*f=\sum_{l+l'\ge j} c_{ll'}(\pi_1^*f_l)(\pi_2^*f'_{l'}) \] 
where $f_l\in \I_l$, $f_{l'}'\in \I_{l'}$, $c_{ll'}\in C^\infty(G^{(2)})$ and $\pi_1,\pi_2\colon G^{(2)}=G\times_M G\rightarrow G$ are the projection maps to the two factors. Therefore
\[ (X^Lf)(g)=\frac{\partial}{\partial t}\bigg|_0\sum_{l+l'\ge j} c_{ll'}(g,s(g)\exp(tX))f_l(g)f'_{l'}(s(g)\exp(tX)). \]
Applying the Leibniz rule to this expression, the result is a sum of two contributions coming from applying the $t$-derivative to $c_{ll'}$ and to $f'_{l'}$ respectively, and we analyze each contribution in turn. When the derivative $\partial/\partial t|_0$ is applied to $c_{ll'}$ we obtain
\[ \sum_{l+l'\ge j} ((0\times X^L)c_{ll'})(g,s(g))f_l(g)f'_{l'}(s(g)). \]
If $l'>0$ then $f'_{l'}(s(g))=0$. Otherwise if $l'=0$ then $l\ge j$ $\Rightarrow$ $f_l\in \I_j$, and hence this expression belongs to $\I_j\subset \I_{j-k}$. 

When the derivative $\partial/\partial t|_0$ is applied to $f'_{l'}$ we obtain
\begin{equation} 
\label{e:2nd-contrib}
\sum_{l+l'\ge j} c_{ll'}(g,s(g))f_l(g)(X^Lf'_{l'})(s(g)). 
\end{equation}
If $l'=0$ then $l\ge j$ and hence the corresponding term of \eqref{e:2nd-contrib} belongs to $\I_j\subset \I_{j-k}$. Otherwise if $l'>0$ then $X^Lf'_{l'}|_M=\hat{X}f'_{l'}|_M$, therefore the $l,l'$ summand equals
\begin{equation} 
\label{e:llprime-summand}
c_{ll'}(g,s(g))f_l(g)(\hat{X}f'_{l'})(s(g))=c_{ll'}(g,s(g))f_l(g)(s^*u^*(\hat{X}f'_{l'}))(g). 
\end{equation}
The vector field $\hat{X}$ had the property that $\hat{X} \I_{l'}\subset \I_{l'-k}$ for all $l'$, thus $\hat{X}f'_{l'}\in \I_{l'-k}$. Recall $u\colon M\hookrightarrow G$, $s\colon G\rightarrow M$ are both weighted morphisms, hence the composition $u\circ s\colon G\rightarrow G$ is a weighted morphism, and therefore $s^*u^*(\hat{X} f'_{l'})\in \I_{l'-k}$. Since $f_l\in \I_l$ and $\I_l\I_{l'-k}\subset \I_{l+l'-k}$, it follows that \eqref{e:llprime-summand} belongs to $\I_{l+l'-k}\subset \I_{j-k}$. Hence each summand in \eqref{e:2nd-contrib} belongs to $\I_{j-k}$, and this completes the argument that $X^Lf \in \I_{j-k}$.
\bigskip

\noindent b) This follows immediately from a) and the definition of $A_\bullet$-order. 
\bigskip

\noindent c) Let $X\in \A_i$ and $Y\in \A_j$. By part a), $X^L,Y^L$ have $\I_\bullet$-orders at most $i,j$ respectively. Hence $X^LY^L$, $Y^LX^L$ both have $\I_\bullet$-order at most $i+j$, and thus $[X,Y]^L=[X^L,Y^L]$ has $\I_\bullet$-order at most $i+j$. By part a) again, $[X,Y]\in \A_{i+j}$.
\bigskip

\noindent d) Let $\I_\bullet'$ be the weighting defined by the Lie filtration $A_\bullet$ (Theorem \ref{t:wt-assoc-Lie-filt}). We have $\I_1=\I_M=\I_1'$. We must show that $\I_k=\I_k'$ for $k=2,3,...$. Both weightings have the same filtration $A_\bullet$ of $A=\nu(G,M)$, and hence the same weight sequence. Let $f \in \I_k$. Let $P \in \U\A$ be a monomial with $A_\bullet$-order at most $k-1$. By part b), $P^L$ has $\I_\bullet$-order at most $k-1$, hence $P^Lf\in \I_M$; as this holds for all $P$ with $A_\bullet$-order at most $k-1$, $f\in \I_k'$ by Proposition \ref{p:operator-condition}, and therefore $\I_k\subset \I_k'$. In particular if $\phi=(x^1,...,x^n)$ are $\I_\bullet$-weighted local coordinates with weights $w_1,...,w_n$ then $x^i\in \I'_{w_i}$. But this means we have a coordinate chart $\phi=(x^1,...,x^n)$ for which the coordinates have the correct weights (with multiplicities) to be an $\I_\bullet'$-weighted coordinate chart, hence by Proposition \ref{p:criterion-for-wt-chart}, $\phi=(x^1,...,x^n)$ is an $\I_\bullet'$-weighted coordinate chart as well (at least near $M$), and it follows that the two weightings coincide.
\end{proof}

\printbibliography
\end{document}